\documentclass[10pt]{article}
\usepackage{graphicx}
\usepackage[utf8]{inputenc}
\usepackage{amsmath}
\usepackage{combelow}
\usepackage[pdftex,pdfpagelabels,hypertexnames=false]{hyperref}
\usepackage{combelow}
\usepackage{amssymb}
\usepackage{amsthm}
\usepackage{enumerate}
\usepackage{fullpage}
\usepackage{url}
\usepackage{float}
\usepackage{verbatim}
\usepackage{xcolor}
\usepackage{afterpage}
\usepackage{titlesec}

\allowdisplaybreaks   

\usepackage[style=numeric, giveninits=true, maxnames=4]{biblatex}
\theoremstyle{plain}
\newtheorem{thm}{Theorem}[section]
\newtheorem{defi}[thm]{Definition}
\newtheorem{lem}[thm]{Lemma}

\newtheorem{prop}[thm]{Proposition}
\newtheorem{rmk}[thm]{Remark}

\newcommand{\R}{\mathbb{R}}
\newcommand{\N}{\mathbb{N}}

\newcommand{\zer}{{\rm zer}}

\DeclareMathOperator{\Id}{Id}

\DeclareMathOperator{\trace}{tr}
\DeclareMathOperator{\argmin}{argmin}

\title{Long-Time Analysis of Stochastic Heavy Ball Dynamics for Convex Optimization and Monotone Equations}
\author{Radu Ioan Bo\cb{t} \footnote{Faculty of Mathematics, University of Vienna, Oskar-Morgenstern-Platz 1, 1090 Vienna, Austria, e-mail: \url{radu.bot@univie.ac.at}. Research partially supported by the Austrian Science Fund (FWF), project 10.55776/P34922.}  \and Chiara Schindler \footnote{Faculty of Mathematics, University of Vienna, Oskar-Morgenstern-Platz 1, 1090 Vienna, Austria, e-mail: \url{chiara.schindler@univie.ac.at}.}}
\date{}

\begin{document}

\maketitle

\begin{abstract}
In a separable real Hilbert space, we study the problem of minimizing a convex function with Lipschitz continuous gradient in the presence of noisy evaluations. To this end, we associate a stochastic Heavy Ball system, incorporating a friction coefficient, with the optimization problem. We establish existence and uniqueness of trajectory solutions for this system. Under a square integrability condition for the diffusion term, we prove almost sure convergence of the trajectory process to an optimal solution, as well as almost sure convergence of its time derivative to zero. Moreover, we derive almost sure and expected convergence rates for the function values along the trajectory towards the infimal value. Finally, we show that the stochastic Heavy Ball system is equivalent to a Su-Boyd-Cand\`{e}s-type system for a suitable choice of the parameter function, and we provide corresponding convergence rate results for the latter.

In the second part of this paper, we extend our analysis beyond the optimization framework and investigate a monotone equation induced by a monotone and Lipschitz continuous operator, whose evaluations are assumed to be corrupted by noise. As before, we consider a stochastic Heavy Ball system with a friction coefficient and a correction term, now augmented by an additional component that accounts for the time derivative of the operator. We establish analogous convergence results for both the trajectory process and its time derivative, and derive almost sure as well as expected convergence rates for the decay of the residual and the gap function along the trajectory. As a final result, we show that a particular instance of the stochastic Heavy Ball system for monotone equations is equivalent to a stochastic second-order dynamical system with a vanishing damping term. Remarkably, this system exhibits fast convergence rates for both the residual and gap functions.
\end{abstract}


\section{Introduction}
For the minimization problem of a convex differentiable function $f:\mathcal{H}\rightarrow \R$ with $L_{\nabla f}$-Lipschitz continuous gradient, where $\mathcal{H}$ denotes a separable real Hilbert space,
\begin{align}\label{eq:opt}
    \min_{y\in\mathcal{H}} f(y),
\end{align}
and assume that its set of solutions $\argmin f = \{y^* \in\mathcal{H}: f(y^*) = \inf_\mathcal{H} f\}$ is nonempty.  Here,  $ \inf_\mathcal{H} f \in \R$ denotes the infimal value of $f$. 

We attach to this minimization problem the stochastic Heavy Ball with friction system on $[t_0, +\infty)$, where $t_0 \geq 0$,
\begin{align}\tag{SHBF}\label{eq:SHBF}
    \begin{cases}
    dY(t) = X(t) dt
    \\ dX(t) = (-\lambda X(t) - b(t) \nabla f(Y(t))) dt + \sigma_Y(t) dW(t)
    \\ Y(t_0) = Y_0, \ \mbox{} \ X(t_0) = X_0,
    \end{cases}
\end{align}
defined on a filtered probability space $(\Omega,\mathcal{F},\{\mathcal{F}_t\}_{t\geq t_0},\mathbb{P})$, where $W$ denotes a $\mathcal{H}$-valued Brownian motion, the diffusion term, denoted by $\sigma_Y :[t_0,+\infty)\rightarrow \mathcal{L}(\mathcal{H},\mathcal{H})$, is a measurable map, $\lambda >0$  and $b:[t_0,+\infty) \rightarrow (0,+\infty)$ is a continuously differentiable and nondecreasing function.

Our first goal is to establish the existence and uniqueness of trajectory solutions of~\eqref{eq:SHBF} and to investigate their long-time behavior. Assuming that $\sigma_Y$ is square integrable with respect to the Hilbert–Schmidt norm  $\|\cdot\|_{HS}$ on $\mathcal{L}(\mathcal{H},\mathcal{H})$, we derive for $f(Y(t)) - \inf_\mathcal{H} f$ an almost sure (a.s.) convergence rate of $o\left(\frac{1}{b(t)}\right)$ and a convergence rate of $\mathcal{O}\left(\frac{1}{b(t)}\right)$ in expectation as $t \to +\infty$. Moreover, we show that the trajectory process $Y(t)$ converges, almost surely, weakly to a minimizer of $f$, and that $\|X(t)\|$ converges almost surely to $0$ as $t \to +\infty$.

Note that~\eqref{eq:SHBF} is the stochastic counterpart of the deterministic system
\begin{align}\label{HBtimescaling}
    \ddot{y}(t) + \lambda \dot{y}(t) + b(t)\nabla f(y(t)) = 0,
\end{align}
with Cauchy data $y(t_0) := Y_0$ and $\dot{y}(t_0) := X_0$, which has recently been investigated in \cite{attouch-bot-hulett-nguyen--heavy-ball} from the viewpoint of its asymptotic properties. This system as its turn can be seen as the time parametrization of the Heavy Ball dynamics with friction
\begin{align}\label{HBtimescaling-no-b}
    \ddot{y}(t) + \lambda \dot{y}(t) + \nabla f(y(t)) = 0,
\end{align}
first introduced by Polyak in in~\cite{polyak-optimization, polyak-speedingupconvergence}. \'Alvarez later showed in \cite{alvarez} that the trajectory $y(t)$ satisfies $f(y(t)) \to \inf_\mathcal{H} f$ and converges weakly to a minimizer of $f$ as $t \to +\infty$. According to ~\cite{attouch-goudou-redont}, $y(t)$ can be interpreted as the horizontal position of a heavy ball moving alongside the graph of $f$ under viscous friction with coefficient $\lambda$. This interpretation gives the system its well-known name, the ``Heavy Ball dynamics".

Using the time–rescaling technique introduced in the deterministic setting by Attouch, Bo\cb{t} and Ngyuen in~\cite{attouch-bot-nguyen}, we relate the stochastic Heavy Ball dynamics~\eqref{eq:SHBF} to the following dynamics with vanishing damping of Su-Boyd-Cand\`{e}s type on $[s_0, +\infty)$, for $s_0 > 0$ and $\alpha > 3$,
\begin{align}\label{eq:SAVD-alpha}\tag{SAVD}
    \begin{cases}
        dZ(s) = Q(s)ds
        \\ dQ(s) = \left(-\frac{\alpha}{s}Q(s) - \nabla f(Z(s))\right) ds + \sigma_{Z}(s) dW(s)
        \\ Z(s_0) = Z_0, \ \mbox{} \ Q(s_0) = Q_0,
    \end{cases}
\end{align}
and establish asymptotic properties for the latter by transferring them from the original system. In particular,  we demonstrate that $f(Z(s)) - \inf_{\cal H} f =o\left(\frac{1}{s^2}\right)$ almost surely, and $\mathbb{E}\left(f(Z(s)) - \inf_{\cal H} f  \right) = \mathcal{O}\left(\frac{1}{s^2}\right)$, in addition to the almost sure weak convergence of $Z(s)$ to a minimizer of $f$ as $s \to +\infty$. In this way, we recover -- via a different route -- the results obtained by Maulen-Soto, Fadili, Attouch and Ochs in~\cite{mfao-stochasticintertialgradient}. Observe that \eqref{eq:SAVD-alpha} is indeed a stochastic version of the celebrated deterministic Su-Boyd-Cand\`{e}s system \cite{su-boyd-candes}
\begin{align*}
    \ddot{z}(s) + \frac{\alpha}{s}\dot{z}(s) + \nabla f(z(s)) = 0,
\end{align*}
which shows a rate of convergence for $f(z(s)) - \inf_{\cal H}$ of $\mathcal{O}\left(\frac{1}{s^2}\right)$ as $s \to +\infty$, thus outperforming the classical gradient flow, which exhibits a rate of convergence for $f(z(s)) - \inf_{\cal H}$ of $o\left(\frac{1}{s}\right)$ as $s \to +\infty$. Even more, according to \cite{attouch-chbani-peypoquet-riahi}, if $\alpha >3$, then $f(z(s)) - \inf_{\cal H}$ of $o\left(\frac{1}{s^2}\right)$, and the trajectory solution $z(s)$ converges weakly to a minimizer of $f$ as $s \to +\infty$. An explicit discretization of this system leads to the Nesterov accelerated gradient method \cite{nesterov-acc-gradient, chambolle-dossal}, which shares the asymptotic properties of the continuous time dynamics.

It is worth noting that the accelerated convergence of function values observed in the deterministic Su–Boyd–Cand\`{e}s dynamical system, when compared to the classical gradient flow, also manifests in the stochastic setting. In particular, analogous improvements arise for the stochastic differential equation \eqref{eq:SAVD-alpha}, when compared with the stochastic gradient flow
\begin{align*}
    \begin{cases}
        dZ(s) = -\nabla f(Z(s))dt + \sigma_Z(s)dW(s)
        \\ Z(s_0) = Z_0,
    \end{cases}
\end{align*}
that has been explored by Maulen-Soto, Fadili and Attouch in~\cite{soto-fadili-attouch}.

In the second part of the paper, we will generalize the framework to the problem of solving the monotone equation
\begin{align}\label{eq:mon}
    V(y) = 0,
\end{align}
where $V:\mathcal{H}\rightarrow\mathcal{H}$ is a monotone and $L_V$-Lipschitz continuous operator. We assume that the set of zeros of $V$, $\zer V=\{y^* \in \mathcal{H}: V(y^*) = 0\}$, is nonempty. The operator $V$ is said to be monotone if
\begin{align*}
    \langle V(y)-V(x), y-x\rangle \geq 0 \ \mbox{for every} \ x,y\in\mathcal{H}.
\end{align*}
The most illustrative examples of monotone operators include $V=\nabla f$, the gradient of a convex and differentiable function $f : {\mathcal H} \rightarrow \R$, and the resolvent $V = (\Id + M)^{-1} : {\mathcal H} \to {\mathcal H}$ of a (set-valued) maximally monotone operator $M : \mathcal{H} \to 2^\mathcal{H}$, where $\Id$ denotes the identity operator on ${\mathcal H}$. A further example of a monotone operator is the saddle-point operator
$$V:\mathcal{H} \times \mathcal{G} \to \mathcal{H} \times \mathcal{G}, \quad V(x,y) = \left (\nabla_x \Phi(x,y), -\nabla_y \Phi(x,y) \right)$$
of a convex-concave differentiable function $\Phi : \mathcal{H} \times \mathcal{G} \to \R$.

We attach to the monotone equation \eqref{eq:mon} the following stochastic Heavy Ball system on $[t_0, +\infty)$, for $t_0 \geq 0$, 
\begin{align}\tag{SHBFOP}\label{eq:SHBF-op}
    \begin{cases}
        dY(t) = X(t)dt
        \\ dV(Y(t)) = H(t) dt
        \\ dX(t) = \left(-\lambda X(t) - \mu(t) H(t) - \gamma(t) V(Y(t))\right) dt + \sigma_Y(t) dW(t)
          \\ Y(t_0) = Y_0, \ \mbox{} \ X(t_0) = X_0,
    \end{cases}
\end{align}
where $X(\cdot)$, $Y(\cdot)$ and $V(Y(\cdot))$ are adapted stochastic processes with the property that there exists a continuous adapted process $H(\cdot)$ such that $dV(Y(t)) = H(t)dt$ for every $t \geq t_0$. In addition, we assume that $\lambda >0$, $\mu:[t_0,+\infty) \rightarrow (0,+\infty)$ is a continuously differentiable and nondecreasing function, and $\gamma:[t_0,+\infty) \rightarrow (0,+\infty)$ a continuous function. As in the optimization setting, first we establish the existence and uniqueness of trajectory solutions of \eqref{eq:SHBF-op}. Assuming that $\sigma_Y$ is square integrable with respect to the Hilbert–Schmidt norm  $\|\cdot\|_{HS}$ on $\mathcal{L}(\mathcal{H},\mathcal{H})$, we derive for the residual $\|V(Y(t))\|$ and the gap function $\langle Y(t) - y^*, V(Y(t))\rangle$, where $y^* \in \zer V$, an almost sure convergence rate of $o\left(\frac{1}{\mu(t)}\right)$ and a convergence rate of $\mathcal{O}\left(\frac{1}{\mu(t)}\right)$ in expectation as $t \to +\infty$. Moreover, we show that the trajectory process $Y(t)$ converges, almost surely, weakly to a zero of $V$, and that $\|X(t)\|$ converges almost surely to $0$ as $t \to +\infty$.

Note that \eqref{eq:SHBF-op} is the stochastic counterpart of the deterministic system
\begin{align}\label{eq:sysmon}
    \ddot{y}(t) + \lambda \dot{y}(t) + \mu(t) \frac{d}{dt}V(y(t)) + \gamma(t) V(y(t)) = 0
\end{align}
with Cauchy data $y(t_0) := Y_0$ and $\dot{y}(t_0) := X_0$, which has recently been investigated in \cite{attouch-bot-hulett-nguyen--heavy-ball} from the viewpoint of asymptotic properties. The trajectory solutions of \eqref{eq:sysmon} are understood in the strong global sense, meaning that the mapping $t \mapsto V(y(t))$ is differentiable almost everywhere, and the equation holds almost everywhere. The dynamical system \eqref{eq:sysmon} can be seen as a second-order regularization that improves the convergence behavior of
\begin{align}\label{eq:attouch-svaiter}
        \dot{y}(t) + \mu(t) \frac{d}{dt}V(y(t)) + \mu(t) V(y(t)) = 0
\end{align}
with Cauchy data $y(t_0) = Y_0$ , introduced by Attouch and Svaiter in \cite{attouch-svaiter}. For this first-order system, it holds that $\|V(y(t))\| \to 0$ as $t \to +\infty$, and the trajectory $y(t)$ converges weakly to a zero of $V$.

Consider the first-order dynamical system, which directly mirrors the gradient flow approach,
\begin{align}\label{eq:gradient-flow-operator}
    \dot{y}(t) = V(y(t)).
\end{align}
This system does not provide the appropriate model for approaching the solution set of \eqref{eq:mon} when the operator 
$V$ is monotone and Lipschitz continuous. Indeed, in this case only the ergodic trajectory $t \mapsto \frac{1}{t-t_0}\int_{t_0}^{t} y(s)ds$ converges to a zero of $V$ as $t\rightarrow +\infty$ (see~\cite{baillon-brezis}.) In general, the trajectory $y(t)$ itself does not converge, as exemplified by the counterclockwise rotation operator by $\frac{\pi}{2}$ radians in $\R^2$ (see ~\cite{cominetti-peypouquet-sorin}). Convergence of the trajectory $y(t)$
arises only when $V$ is cocoercive; that is, there exists $\rho>0$ such that
\begin{align*}
    \langle V(y)-V(x), y-x\rangle \geq \rho \|V(y)-V(x)\|^2 \ \mbox{for every} \ x,y\in\mathcal{H}.
\end{align*}
In this case, $y(t)$ converges weakly to a zero of $V$, while the residual $\|V(y(t))\|$ decays to zero with rate $o\left(\frac{1}{\sqrt{t}}\right)$ as $t \to +\infty$ (see \cite{attouch-bot-nguyen}).

By applying the same time–rescaling technique as in the convex minimization setting, we connect \eqref{eq:SHBF-op} with the following dynamics on $[s_0, +\infty)$, where $s_0 > 0$ and $\alpha >2$ and $\beta >0$,
\begin{align}\label{eq:SFOGDA}\tag{SFOGDA}
    \begin{cases}
        dZ(s) = Q(s) ds\\
        dV(Z(s)) = K(s) ds\\
        dQ(s) = \left(-\frac{\alpha}{s}Q(s) - \beta K(s) - \frac{\alpha \beta}{2s} V(Z(s))\right) + \sigma_Z(s) dW(s).
    \end{cases}
\end{align}
We then establish the asymptotic behavior of \eqref{eq:SFOGDA} by transferring properties from the original system. In particular, we show that $\|V(Z(s))\| = o\left(\frac{1}{s}\right)$ and  $\langle Z(s) - z^*, V(z(s)) \rangle = o\left(\frac{1}{s}\right)$ almost surely, $\mathbb{E}(\|V(Z(s))\|) = \mathcal{O}\left(\frac{1}{s}\right)$ and $\mathbb{E}(\langle Z(s) - z^*, V(Z(s)) \rangle) = \mathcal{O}\left(\frac{1}{s}\right)$, for every $z^* \in \zer V$, in addition to the almost sure weak convergence of $Z(s)$ to a zero of $V$ as $s \to +\infty$.

Equation \eqref{eq:SFOGDA} can be seen as a stochastic counterpart of the Fast Optimistic Gradient Descent Ascent (OGDA) system 
\begin{align}\label{eq:fogda}
            \ddot{z}(s) + \frac{\alpha}{s}\dot{z}(s) + \beta \frac{d}{ds}V(z(s)) + \frac{\alpha \beta}{2s} V(z(s)) = 0,
\end{align}
recently introduced by Bo\cb{t}, Csetnek and Nguyen in~\cite{bot-csetnek-nguyen} in connection with solving monotone equations \eqref{eq:mon}. Beyond the weak convergence of trajectories to a zero of $V$, this system achieves convergence rates of order $o\left(\frac{1}{s}\right)$ for both the residual $\|V(z(s))\|$ and the gap function $\langle z(s) - z^*, V(z(s)) \rangle$, for every $z^* \in \zer V$, as $s \to +\infty$. The Fast OGDA algorithm, obtained via an explicit discretization of \eqref{eq:fogda}, inherits these asymptotic guarantees. Remarkably, both the continuous- and discrete-time formulations of Fast OGDA currently yield the best known convergence rates for monotone equations in the literature.

It is worth noting that in~\cite{bot-s} we have investigated the following stochastic dynamical system (in a slightly more general form) on $[0,+\infty)$ -- a stochastic counterpart of \eqref{eq:attouch-svaiter} in the finite-dimensional setting
\begin{align}\label{eq:botschindler}
        d\big(Y(t) + \mu(t)V(Y(t))\big) = -(\gamma(t)-\dot{\mu}(t))V(Y(t))dt + \sigma_Y(t)dW(t),
\end{align}
where $\mu : [0, +\infty) \to (0,+\infty)$ and $\gamma : [0, +\infty) \to (0,+\infty)$ are bounded away from zero, with $\mu$ continuous differentiable and nonincreasing,  and $\gamma$ integrable. Assuming the diffusion term is square integrable, we have established ergodic convergence rates in expectation for both the norm square of the residual and the gap function. We emphasize, however, that the asymptotic behavior of \eqref{eq:fogda} significantly outperforms that of \eqref{eq:botschindler}.

\section{The case of convex minimization problem}\label{sec:convex-opt}

In this section we will investigate the stochastic Heavy Ball dynamics \eqref{eq:SHBF} 
\begin{align*}
    \begin{cases}
    dY(t) = X(t) dt
    \\ dX(t) = (-\lambda X(t) - b(t) \nabla f(Y(t))) dt + \sigma_Y(t) dW(t)
    \\ Y(t_0) = Y_0, \ \mbox{} \ X(t_0) = X_0
    \end{cases}
\end{align*}
in relation to the convex minimization problem \eqref{eq:opt}.

\subsection{Long-time analysis of \eqref{eq:SHBF}}

We begin by establishing the existence and uniqueness of trajectory solutions to \eqref{eq:SHBF}, drawing on results from the theory of stochastic differential equations. The relevant theorems and definitions are collected in the appendix.
\begin{thm}
If $\|\sigma_Y(t)\|_{HS}\leq\sigma_*$ for every $t\geq t_0$, then \eqref{eq:SHBF} has a unique strong solution $(Y,X)$. In addition, $(Y,X) \in S^{\nu}_{\cal H \times \cal H}$ for every $\nu \geq 2$.
\end{thm}
\begin{proof}
Let $t_0 \leq T$. First, we prove that condition $(1)$ in Theorem~\ref{thm:existence-uniqueness-solutions} (i) is satisfied. Let $Y,X \in C([t_0, T], \mathcal{H})$, $t_0 \leq t \leq T$ and $\omega\in\Omega$ be fixed.  For $y^* \in \argmin f$, it holds
    \begin{align*}
        & \ \left\|\begin{pmatrix}
            X(\omega, t)
            \\ -\lambda X(\omega, t) - b(t)\nabla f(Y(\omega, t))
        \end{pmatrix}\right\| + \|\sigma_Y(t)\|_{HS}\\
= & \ \sqrt{\|X(\omega, t)\|^2 + \|\lambda X(\omega, t) + b(t)\nabla f(Y(\omega, t))\|^2} + \|\sigma_Y(t)\|_{HS}\\
\leq & \ \sqrt{\|X(\omega, t)\|^2(1+2\lambda^2) + 2b(t)^2 L_{\nabla f}^2\|Y(\omega, t)-y^*\|^2} + \|\sigma_Y(t)\|_{HS}\\
\leq & \ \sqrt{\|X(\omega,t)\|^2(1+2\lambda^2) + 4b(t)^2 L_{\nabla f}^2\|Y(\omega,t)\| + 4b(t)^2 L_{\nabla f}^2\|y^*\|^2} + \|\sigma_Y(t)\|_{HS}
\\ \leq & \ \left(\max\left(\sqrt{1+2\lambda^2}, 2 L_{\nabla f} \sup_{t_0\leq t\leq T} b(t), 2 L_{\nabla f} \|y^*\| \sup_{t_0\leq t\leq T} b(t)\right) + \sigma_* \right)\left( 1+\sup_{t_0\leq s\leq T} \left\| \begin{pmatrix}
            X(\omega, s)
            \\ Y(\omega, s)
        \end{pmatrix}\right\|\right).
    \end{align*}
Condition $(2)$ in Theorem~\ref{thm:existence-uniqueness-solutions} (i) is also satisfied. Indeed, let $Y_1,X_1,Y_2,X_2\in C([t_0,T],\mathcal{H}), t_0\leq t\leq T$ and $\omega\in\Omega$ be fixed. It holds
    \begin{align*}
        & \ \left\| \begin{pmatrix}
            X_1(\omega, t)
            \\ - \lambda X_1(\omega, t) - b(t)\nabla f(Y_1(\omega, t))
        \end{pmatrix} - \begin{pmatrix}
            X_2(\omega,t)
            \\ -\lambda X_2(\omega,t) - b(t) \nabla f(Y_2(\omega,t))
        \end{pmatrix}\right\| + \| \sigma_Y(t) - \sigma_Y(t)\|_{HS}\\
        = & \ \sqrt{\| X_1(\omega,t)-X_2(\omega,t)\|^2 + \|\lambda(X_1(\omega,t)-X_2(\omega,t)) + b(t)(\nabla f(Y_1(\omega,t)) - \nabla f(Y_2(\omega,t)))\|^2}\\
        \leq & \ \sqrt{(1+2\lambda^2)\|X_1(\omega,t)-X_2(\omega,t)\|^2 + 2b(t)^2 L_{\nabla f}^2 \|Y_1(\omega,t)-Y_2(\omega,t)\|^2}\\
        = & \ \max\left(\sqrt{1+2\lambda^2}, \sqrt{2}L_{\nabla f} \sup_{t_0\leq t\leq T} b(t)\right) \max_{t_0 \leq s \leq T}\left\| \begin{pmatrix}
            X_1(\omega,s)
            \\Y_1(\omega,s)
        \end{pmatrix} - \begin{pmatrix}
            X_2(\omega,s)
            \\Y_2(\omega,s)
        \end{pmatrix}\right\|.
    \end{align*}
According to Theorem~\ref{thm:existence-uniqueness-solutions} (i), \eqref{eq:SHBF} has a unique strong solution on 
$[t_0,T]$. Since $t_0 \leq T$ has been arbitrarily chosen, the existence of a unique solution of \eqref{eq:SHBF} on $[t_0, +\infty)$ follows by standard extension arguments. 

By using the same estimates as in the verification of condition $(2)$, one easily sees, by using Theorem~\ref{thm:existence-uniqueness-solutions} (ii), that for every $\nu \geq 2$  the solution $(Y,X)$ of \eqref{eq:SHBF} lies in  $S^{\nu}_{\cal H \times \cal H}[t_0,T]$ for every $t_0 \leq T$, and therefore in $S^{\nu}_{\cal H \times \cal H}$.
\end{proof}

For the long-time analysis of the trajectory solution of \eqref{eq:SHBF}, we shall assume that
\begin{align}\label{eq:assumption-b}
    \sup_{t\geq t_0} \tfrac{\dot{b}(t)}{b(t)} < \lambda.
\end{align}

\begin{thm}\label{thm:convergence-shbf}
Let $\sigma_Y$ be square integrable, and let $(Y,X) \in S^{2}_{\cal H \times \cal H}$ be a strong solution of~\eqref{eq:SHBF}. Then, the following statements are true:
    \begin{enumerate}[(i)]
        \item $f(Y(t)) - \inf_\mathcal{H} f = o\left(\tfrac{1}{b(t)}\right)$ and $\|X(t)\| = o(1)$ as $t\rightarrow +\infty$ \mbox{a.s.};
        \item $\mathbb{E} \left(f(Y(t)) - \inf_\mathcal{H} f\right)  = \mathcal{O}\left(\tfrac{1}{b(t)}\right)$ and $\mathbb{E}\left(\|X(t)\|^2\right) = \mathcal{O}\left(1\right)$ as $t\rightarrow +\infty$;
        \item $Y(t)$ converges weakly to a minimizer of $f$ as $t\rightarrow +\infty$ a.s.
    \end{enumerate}
\end{thm}

\begin{proof}
Choose $y^* \in \argmin f$ and $\eta$ such that
    \begin{align*}
        \sup_{t\geq t_0} \frac{\dot{b}(t)}{b(t)} < \eta < \lambda.
    \end{align*}
Consider the energy function
    \begin{align*}
        \mathcal{E}_\eta(t,y,x) := b(t)\left(f(y) - \inf\nolimits_{\mathcal{H}} f\right) + \frac{1}{2} \|\eta(y-y^*) + x\|^2 + \frac{1}{2}\eta(\lambda-\eta)\|y-y^*\|^2.
    \end{align*}
We have
    \begin{align*}
        \frac{d}{dt} \mathcal{E}_\eta(t,y,x) & = \dot{b}(t)(f(y) - \inf\nolimits_\mathcal{H} f)
        \\ \frac{d}{dy}\mathcal{E}_\eta(t,y,x) & = b(t)\nabla f(y) + \eta(\eta(y-y^*)+x) + \eta(\lambda-\eta)(y-y^*)
        \\ \frac{d}{dx} \mathcal{E}_\eta(t,y,x) & = \eta(y-y^*) + x
        \\ \frac{d^2}{dx^2}\mathcal{E}_\eta(t,y,x) & = \Id,
    \end{align*}
so, by It\^{o}'s formula,
    \begin{align*}
        & \ d\mathcal{E}_\eta(t,Y(t),X(t)) \\
        = & \ \Big(\dot{b}(t)(f(Y(t)) - \inf\nolimits_\mathcal{H} f) + \langle b(t)\nabla f(Y(t)) + \eta^2(Y(t)-y^*) + \eta X(t) + \eta(\lambda - \eta)(Y(t)-y^*), X(t)\rangle\\
        &\quad + \langle \eta(Y(t)-y^*) + X(t), -\lambda X(t)-b(t)\nabla f(Y(t)) \rangle + \trace\left(\sigma_Y(t)\sigma_Y(t)^*\right)\Big) dt\\
        & \  + \langle \eta(Y(t)-y^*) + X(t), \sigma_Y(t) dW(t)\rangle\\
    = & \ \Big(\dot{b}(t)(f(Y(t)) - \inf\nolimits_\mathcal{H} f) + \langle b(t)\nabla f(Y(t)) + \eta X(t) + \eta\lambda(Y(t)-y^*), X(t)\rangle -\eta\lambda \langle Y(t)-y^*, X(t)\rangle\\
        &\quad - \eta b(t)\langle Y(t)-y^*, \nabla f(Y(t))\rangle - \lambda\|X(t)\|^2 - b(t) \langle X(t), \nabla f(Y(t))\rangle + \|\sigma_Y(t)\|^2_{HS}\Big) dt\\
        & \ + \langle \eta(Y(t)-y^*) + X(t), \sigma_Y(t) dW(t)\rangle\\
    = & \ \left(\dot{b}(t)(f(Y(t)) - \inf\nolimits_\mathcal{H} f) - \eta b(t)\langle \nabla f(Y(t)), Y(t)-y^*\rangle + (\eta-\lambda)\|X(t)\|^2 \right)dt\\
        & \ + \|\sigma_Y(t)\|^2_{HS} dt + \langle \eta(Y(t)-y^*) + X(t), \sigma_Y(t) dW(t)\rangle \quad \forall t \geq t_0.
    \end{align*}
By using the convexity of $f$, for every $s \geq t_0$, we have
    \begin{align*}
       u(s):= & \ \dot{b}(s)(f(Y(s)) - \inf\nolimits_\mathcal{H} f) - \eta b(s) \langle \nabla f(Y(s)), Y(s)-y^*\rangle \\
       \leq & \ (\dot{b}(s) - \eta b(s))(f(Y(s)) - \inf\nolimits_\mathcal{H} f) \leq 0.
    \end{align*}
Thus, for every $t \geq t_0$
    \begin{align}\label{eq:ito-formula-energy-optimization}
        \mathcal{E}_\eta(t,Y(t),X(t)) = & \ \mathcal{E}_\eta(t_0,Y(t_0),X(t_0)) \nonumber \\
        & + \int_{t_0}^t \left(\dot{b}(s)(f(Y(s)) - \inf\nolimits_\mathcal{H} f) - \eta b(s) \langle \nabla f(Y(s)), Y(s)-y^*\rangle + (\eta - \lambda)\|X(s)\|^2 \right) ds \nonumber \\
        & + \int_{t_0}^t \|\sigma_Y(s)\|^2_{HS} ds + \int_{t_0}^t \langle \eta(Y(s)-y^*) + X(s), \sigma_Y(s) dW(s)\rangle\\
        = & \ \mathcal{E}_\eta(t_0,Y(t_0),X(t_0)) - \underbrace{\int_{t_0}^t (\lambda-\eta)\|X(s)\|^2 - u(s) ds}_{=:U(t)} \nonumber\\
        & \ + \underbrace{\int_{t_0}^t \|\sigma_Y(s)\|^2_{HS} ds}_{=:A(t)} + \underbrace{\int_{t_0}^t \langle \eta(Y(s)-y^*) + X(s), \sigma_Y(s) dW(s)\rangle}_{=:N(t)},\nonumber
    \end{align}
    where $A(\cdot)$ and $U(\cdot)$ are increasing processes fulfilling $U(t_0)=0=A(t_0)$ a.s. According to  Proposition~\ref{prop:ito-formula}, $N(\cdot)$ is a martingale, since $(Y,X) \in S^{2}_{\cal H \times \cal H}$ and for every $t \geq t_0$
    \begin{align*}
        & \ \mathbb{E}\left( \int_{t_0}^t \|\sigma_Y^*(s)(\eta(Y(s)-y^*) + X(s))\|^2 ds \right)\\
        \leq & \ \mathbb{E}\left(\int_{t_0}^t 3\|\sigma_Y(s)\|^2 \left(\eta^2 (\|Y(s)\|^2+\|y^*\|^2) + \|X(s)\|^2 \right) ds\right)\\
        \leq & \ 3\mathbb{E}\left(\left(\eta^2 \sup_{s\in[t_0,t]} \|Y(s)\|^2 + \sup_{s\in[t_0,t]} \|X(s)\|^2 + \eta^2\|y^*\|^2\right)\right) \int_{t_0}^t \|\sigma_Y(s)\|^2 ds <+\infty.
    \end{align*}    
From the square-integrability of $\sigma_Y$, we obtain the existence of $\lim_{t\rightarrow +\infty} A(t) <+\infty$. Therefore, by Theorem \ref{thm:A.9-in-paper}, the limits $\lim_{t\rightarrow +\infty} \mathcal{E}_\eta(t,Y(t),X(t))$ and $\lim_{t\rightarrow +\infty} U(t)$ exist and are finite a.s. 

Taking the expectation in~\eqref{eq:ito-formula-energy-optimization}, we obtain, for every $t\geq t_0$, that
    \begin{align*}
        \mathbb{E}\left(\mathcal{E}_\eta(t,Y(t),X(t)) \right) &\leq \mathbb{E}(\mathcal{E}_\eta(t_0,Y(t_0),X(t_0))) + \mathbb{E}\left(\int_{t_0}^{t} \|\sigma_Y(s)\|^2_{HS} ds\right)
        \\&\leq \mathbb{E}(\mathcal{E}_\eta(t_0,Y(t_0),X(t_0)) + \mathbb{E}\left(\int_{t_0}^{+\infty} \|\sigma_Y(s)\|^2_{HS}ds\right) < +\infty,
    \end{align*}
which means, in particular, that
    \begin{align*}
\sup_{t \geq t_0} \mathbb{E}\left(\|Y(t)-y^*\|^2\right) <+\infty,
    \end{align*}
    \begin{align*}
\sup_{t \geq t_0}\mathbb{E}\left(b(t)\left(f(Y(t)) - \inf\nolimits_\mathcal{H} f\right)\right) < +\infty, \quad  \mbox{and} \quad \sup_{t \geq t_0} \mathbb{E}\left(\|\eta(Y(t)-y^*) + X(t)\|^2\right) <+\infty.
    \end{align*}
Since, for every $t\geq t_0$,
    \begin{align*}
       \|X(t)\|^2  \leq 2 \|\eta(Y(t) - y^*) + X(t)\|^2 + 2 \eta^2 \|Y(t)-y^*\|^2,
    \end{align*}
    we have that
    \begin{align*}
\sup_{t \geq t_0} \mathbb{E}\left(\|X(t)\|^2\right) <+\infty.
    \end{align*}
Therefore, 
    \begin{align*}
       \sup_{t \geq t_0} \mathbb{E}\left(b(t)\left(f(Y(t)) - \inf\nolimits_\mathcal{H} f\right) + \|X(t)\|^2\right) < +\infty,
    \end{align*}
in other words,
\begin{align*}
        \mathbb{E}\Big(b(t)\left(f(Y(t)) - \inf\nolimits_\mathcal{H} f\right) + \|X(t)\|^2\Big) = \mathcal{O}\left(1\right) \ \mbox{as} \ t \to +\infty,
\end{align*}
thus proving $(ii)$.

Further, it follows from
    \begin{align*}
        \lim_{t\rightarrow +\infty} U(t) <+\infty \ \mbox{a.s.}
    \end{align*}
that
    \begin{align*}
    \int_{t_0}^{+\infty} \|X(t)\|^2 dt < +\infty \quad \mbox{and} \quad  \int_{t_0}^{+\infty} (-u(t)) dt < +\infty \ \mbox{a.s.}
    \end{align*}
From here, since
    \begin{align*}
\int_{t_0}^{+\infty}  b(t)\left(\eta - \sup_{t\geq t_0}\frac{\dot{b}(t)}{b(t)}\right) (f(Y(t)) - \inf\nolimits_\mathcal{H} f) dt   \leq & \  \int_{t_0}^{+\infty} (\eta b(t) - \dot{b}(t))(f(Y(t)) - \inf\nolimits_\mathcal{H} f) dt\\
\leq & \ \int_{t_0}^{+\infty} (-u(t)) < +\infty \ \mbox{a.s.},
    \end{align*}
it follows
    \begin{align*}
        \int_{t_0}^{+\infty} b(t)(f(Y(t)) - \inf\nolimits_\mathcal{H} f) dt < +\infty \ \mbox{a.s.}
    \end{align*}
    
Now we choose $\eta_1, \eta_2 \in \left( \sup_{t\geq t_0} \frac{\dot{b}(t)}{b(t)},  \lambda \right), \eta_1 \neq \eta_2$. For every $t \geq t_0$ it holds
    \begin{align*}
        \mathcal{E}_{\eta_1}(t,Y(t),X(t)) - \mathcal{E}_{\eta_2}(t,Y(t),X(t)) &= \frac{1}{2}\lambda(\eta_1-\eta_2)\|Y(t)-y^*\|^2 + (\eta_1-\eta_2)\langle Y(t)-y^*, X(t)\rangle
        \\&= (\eta_1-\eta_2) \left(\frac{\lambda}{2} \|Y(t)-y^*\|^2 + \langle Y(t)-y^*, X(t)\rangle\right).
    \end{align*}
Denoting
    \begin{align*}
        p(t) := \frac{\lambda}{2} \|Y(t)-y^*\|^2 + \langle Y(t)-y^*, X(t)\rangle,
    \end{align*}
 it follows from the existence of $\lim_{t\rightarrow +\infty}\mathcal{E}_{\eta_1}(t,Y(t),X(t)) - \mathcal{E}_{\eta_2}(t,Y(t),X(t))$ a.s. that also $\lim_{t\rightarrow +\infty} p(t) < +\infty$ must exist a.s. Observe that for every $t \geq t_0$ 
    \begin{align*}
        \mathcal{E}_\eta(t,Y(t),X(t)) = & \ b(t)(f(Y(t)) - \inf\nolimits_\mathcal{H} f) + \frac{1}{2}\|\eta(Y(t)-y^*) + X(t)\|^2 + \frac{1}{2}\eta(\lambda-\eta)\|Y(t)-y^*\|^2\\
        = & \ b(t)(f(Y(t))-\inf\nolimits_\mathcal{H} f) + \frac{1}{2}\|X(t)\|^2 + \eta p(t).
    \end{align*}
    Denoting
    \begin{align*}
        h(t) := b(t)(f(Y(t)) - \inf\nolimits_\mathcal{H} f) + \frac{1}{2}\|X(t)\|^2,
    \end{align*}
it yields that the limit
    \begin{align*}
        \lim_{t\rightarrow +\infty} h(t) < +\infty \ \mbox{exists a.s.}
    \end{align*}
Since
    \begin{align*}
        \int_{t_0}^{+\infty} h(t)dt = \int_{t_0}^{+\infty} b(t)(f(Y(t))-\inf\nolimits_\mathcal{H} f) dt + \frac{1}{2}\int_{t_0}^{+\infty} \|X(t)\|^2 dt < +\infty,
    \end{align*}
we can conclude that
    \begin{align*}
        \lim_{t\rightarrow +\infty} h(t) = 0 \ \mbox{a.s.}
    \end{align*}
Thus,
    \begin{align*}
        f(Y(t)) - \inf\nolimits_\mathcal{H} f = o\left(\tfrac{1}{b(t)}\right)  \ \mbox{and} \ \lim_{t\rightarrow +\infty} \|X(t)\| = 0 \ \mbox{a.s.},
    \end{align*}
which proves $(i)$.

Let us proceed to the proof of statement $(iii)$. For $y^* \in \argmin f$ fixed, we have
\begin{align*}
        d(\|Y(t)-y^*\|^2) = 2 \langle Y(t)-y^*, X(t) \rangle dt.
\end{align*}
Applying Lemma~\ref{lem:aux-convergence}, we deduce that the existence of $\lim_{t \to +\infty} p(t) < +\infty$  implies that $\lim_{t\rightarrow +\infty} \|Y(t)-y^*\|^2 < +\infty$ also exists. Hence, there exist subsets $\Omega_1,\Omega_{y*} \subseteq \Omega$ of full measure such that
    \begin{align*}
        f(Y(\omega,t)) - \inf\nolimits_\mathcal{H} f = o\left(\tfrac{1}{b(t)}\right) \ \mbox{} \forall \omega\in\Omega_1 \quad \mbox{and} \quad \lim_{t\rightarrow +\infty} \|Y(\omega,t) - y^*\|^2 < +\infty \ \mbox{} \forall \omega\in\Omega_{y^*}.
    \end{align*}
Before concluding our argument, we must establish the existence of a set of full measure, independent of $y^* \argmin f$, that can take the place of $\Omega_{y^*}$ in the second property. This can be attained using an argument from~\cite{soto-fadili-attouch, bot-s}, which relies on the separability assumption on $\mathcal{H}$. Since $\argmin f$ is closed, the separability of $\mathcal{H}$ ensures the existence of a countable set $S\subseteq \mathcal{H}$ that is dense in $\argmin f$. By the countability of $S$,
    \begin{align*}
        \mathbb{P}\left(\bigcap_{s\in S} \Omega_s\right) = 1 - \mathbb{P}\left(\bigcup_{s\in S} \Omega_s^c\right) \geq 1 - \sum_{s\in S} \mathbb{P}(\Omega_s^c) = 1.
    \end{align*}
    Then, $\mathbb{P}(\bigcap_{s\in S}\Omega_s \cap \Omega_1) = 1$. 
    
    Now let $\omega\in\bigcap_{s\in S} \Omega_s \cap \Omega_1$ be fixed and take $y^*\in\argmin f$. Then, there exists a sequence $(s^k)_{k\geq 0} \subseteq S$ such that $s^k \rightarrow y^*$ as $k\rightarrow +\infty$. For every $k\geq 0$, since $\omega \in \Omega_{s^k}$, it holds
    \begin{align*}
        \lim_{t\rightarrow +\infty} \|Y(\omega,t) - s^k\|< +\infty \ \mbox{exists.}
    \end{align*}
    Using the triangle inequality, we obtain
    \begin{align*}
        \Big| \|Y(\omega,t) - s^k\| - \|Y(\omega,t) - y^*\| \Big | \leq \|s^k - y^*\| \quad \forall k\geq 0 \ \forall t\geq 0.
    \end{align*}
    Hence, for every $k\geq 0$, it holds
    \begin{align*}
        -\|s^k - y^*\| + \lim_{t\rightarrow +\infty} \|Y(\omega,t) - s^k\| &\leq \liminf_{t\rightarrow +\infty} \|Y(\omega,t) - y^*\|
        \\& \leq \limsup_{t\rightarrow +\infty} \|Y(\omega,t) - y^*\| \leq \lim_{t\rightarrow +\infty} \|Y(\omega,t) - s^k\| + \|s^k - y^*\|.
    \end{align*}
Letting $k\rightarrow +\infty$, it follows
    \begin{align*}
        \lim_{t\rightarrow +\infty} \|Y(\omega,t) - y^*\| = \lim_{k\rightarrow+\infty} \lim_{t\rightarrow +\infty} \|Y(\omega,t) - s^k\| < +\infty.
    \end{align*}
Recall that there exists a set $\Omega_\text{cont}\subseteq \Omega$ of full measure such that $Y(\omega, \cdot)$ is continuous for every $\omega\in\Omega_\text{cont}$. Define $\Omega_\text{converge} := \Omega_1 \cap \Omega_\text{cont} \cap \left(\bigcap_{s\in S} \Omega_s \right)$, which is also a subset of full measure of $\Omega$. 

We will use Opial's Lemma (\cite{opial}) to prove the weak almost sure convergence of the trajectory process to a minimizer of $f$. To this end, we fix $\omega\in\Omega_\text{converge}$. From the above, for every $y^* \in \argmin f$, the limit  $\lim_{t\rightarrow +\infty} \|Y(\omega,t) - y^*\| < +\infty$ exists. This verifies the first condition of Opial's Lemma. Using the continuity of $Y(\omega, \cdot)$, we can also show that $Y(\omega, \cdot)$ is  bounded. If $\bar y(\omega)$ is a weak limit point of this trajectory process, then using (i) and the weak lower semicontinuity of $f$, we get that $f(\bar y(\omega)) = \inf_\mathcal{H} f$. Therefore, $\bar y(\omega) \in \argmin f$. This shows that the second condition of Opial's Lemma is also satisfied. In conclusion,  $Y(\omega, t)$ converges weakly to a minimizer of $f$ as $t \to +\infty$.
\end{proof}

\begin{rmk}{\rm
Note that in order to use It\^{o}'s formula in the above proof, we do need to require that, for every $t \geq t_0$, the mapping $(y,x) \mapsto \mathcal{E}_\eta(y,x)$ is twice differentiable in $y$. This is because the correspoding components in $\nabla_{(y,x)} \mathcal{E}_\eta(t,y,x)$ and $\nabla_{(y,x)}^2\mathcal{E}_\eta(t,y,x)$ are multiplied by $0$ in the integral representation of $\mathcal{E}_\eta(t,Y(t),X(t))$ -- in the proof of It\^{o}'s formula, the assumptions are only used for precisely the components that are not multiplied by $0$, so it is sufficient to have the requirements verified only on them.}
\end{rmk}

Depending on the choice for the parameter function $b$, these convergence rates may still leave room for improvement--an issue explored in the theorem below and in the remark immediately following the proof. 

\begin{thm}\label{thm:convergence-rates-shbf}
Assume $\sigma_Y$ is  square integrable and $\displaystyle \int_{t_0}^{+\infty} \frac{1}{b(s)} \left(\int_{\frac{s+t_0}{2}}^s b(r)dr \right) \|\sigma_Y(s)\|_{HS}^2 ds < +\infty$. If $(Y,X) \in S^{2}_{\cal H \times \cal H}$ is a strong solution of~\eqref{eq:SHBF}, then
  \begin{align*}
        f(Y(t)) - \inf\nolimits_\mathcal{H} f = \mathcal{O}\left(\frac{1}{\int_{\frac{t+t_0}{2}}^{t} b(s)ds}\right) \quad \mbox{and} \quad \|X(t)\| = \mathcal{O}\left(\sqrt{\frac{b(t)}{\int_{\frac{t+t_0}{2}}^{t} b(s)ds}}\right) \ \mbox{as} \ t\rightarrow +\infty \ \mbox{a.s.}
    \end{align*}
\end{thm}
\begin{proof}
Define
\begin{align*}
    \Phi(t,y,x) := (f(y)-\inf\nolimits_\mathcal{H} f) + \frac{1}{2b(t)}\|x\|^2,
\end{align*}
and set
\begin{align*}
        \widetilde{\Phi}(t,y,x) := g(t)\Phi(t,y,x),
    \end{align*}
where
\begin{align*}
        g(t) := \int_{\frac{t+t_0}{2}}^t b(s)ds.
\end{align*}
    Observe that
    \begin{align*}
        &\frac{d}{dt} \widetilde{\Phi}(t,y,x) = -\frac{\dot{b}(t)}{2b(t)}\|x\|^2 g(t) + \left(b(t)-b\left(\frac{t+t_0}{2}\right)\right)\Phi(t,x,y)
        \\& \frac{d}{dy}\widetilde{\Phi}(t,y,x) = g(t) \nabla f(y)
        \\& \frac{d}{dx}\widetilde{\Phi}(t,y,x) = \frac{g(t)}{b(t)} x
        \\& \frac{d^2}{dx^2} \widetilde{\Phi}(t,y,x) = \frac{g(t)}{b(t)} \Id.
    \end{align*}
By It\^{o}'s formula,
    \begin{align*}
        & \ d\widetilde{\Phi}(t,Y(t),X(t))\\
        = & \ g(t) \left(-\frac{\dot{b}(t)}{2b(t)^2}\|X(t)\|^2 + \langle \nabla f(Y(t)), X(t) \rangle + \frac{1}{b(t)}\langle X(t), -\lambda X(t) - b(t)\nabla f(Y(t))\rangle\right) dt\\
        & + \left(b(t)-b\left(\frac{t+t_0}{2}\right)\right)\Phi(t,X(t),Y(t)) dt + \left\langle \frac{g(t)}{b(t)}X(t), \sigma_Y(t) dW(t)\right\rangle + \frac{g(t)}{2b(t)} \trace\left(\sigma_Y(t)\sigma_Y(t)^*\right) dt \\
        = & \ -\frac{g(t)}{b(t)} \left(\frac{\dot{b}(t)}{2b(t)} + \lambda\right) \|X(t)\|^2 dt - b\left(\frac{t+t_0}{2}\right)\Phi(t,X(t),Y(t))dt + b(t)\Phi(t,X(t),Y(t)) dt\\
        & \ + \frac{g(t)}{b(t)} \left\langle \sigma_Y(t)^* X(t), dW(t)\right\rangle + \frac{g(t)}{2b(t)}\trace\left(\sigma_Y(t)\sigma_Y(t)^*\right) dt \quad \forall t \geq t_0,
    \end{align*}
therefore,
    \begin{align*}
\widetilde{\Phi}(t,Y(t),X(t)) = & \ \widetilde{\Phi}(0,Y(0),X(0))\\
        & \ - \underbrace{\int_{t_0}^t \frac{g(s)}{b(s)}\left(\frac{\dot{b}(s)}{2b(s)} + \lambda\right)\|X(s)\|^2 + b\left(\frac{s+t_0}{2}\right)\Phi(s,X(s),Y(s)) ds}_{=:U(t)}\\
        & \ + \underbrace{\int_{t_0}^t \frac{g(s)}{2b(s)}\trace(\sigma_Y(s)\sigma_Y(s)^*) + b(s)\Phi(s,X(s),Y(s))ds}_{=:A(t)}\\
        & \ + \underbrace{\int_{t_0}^t \frac{g(s)}{b(s)} \left\langle \sigma_Y(s)^*X(s), dW(s)\right\rangle}_{=:N(t)} \quad \forall t \geq t_0.
    \end{align*}
 $A(\cdot)$ and $U(\cdot)$ are increasing processes fulfilling $U(t_0)=0=A(t_0)$ a.s. Since $(Y,X) \in S^{2}_{\cal H \times \cal H}$ and for every $t \geq t_0$
\begin{align*}
        \mathbb{E} \left(\int_{t_0}^t \left\|\frac{g(s)}{b(s)} \sigma_Y(s)^*X(s)\right\|^2 ds \right) \leq & \ \mathbb{E}\left( \int_{t_0}^T \left(\frac{g(s)}{b(s)}\right)^2 \|\sigma_Y(s)\|^2 \|X(s)\|^2 ds \right)\\
        \leq & \ \sup_{t_0\leq s \leq t}\left(\frac{g(s)}{b(s)}\right)^2\mathbb{E}\left(\sup_{t_0\leq s\leq t} \|X(s)\|^2\right)\int_{t_0}^t\|\sigma_Y(s)\|_{HS}^2 ds < +\infty,
    \end{align*}
taking into account Proposition~\ref{prop:ito-formula},  $N(\cdot)$ is a martingale. We have seen in the proof of Theorem~\ref{thm:convergence-shbf} that 
$\int_{t_0}^{+\infty} b(t)(f(Y(t))-\inf_\mathcal{H})dt <+\infty$ and 
$ \int_{t_0}^{+\infty} \|X(t)\|^2 dt <+\infty$, hence $\int_{t_0}^{+\infty}b(t)\Phi(t,Y(t),X(t)) dt <+\infty$. This, together with the second integrability assumption in the hypothesis, guarantees the existence of $\lim_{t\rightarrow +\infty} A(t)<+\infty$. Thus, from Theorem~\ref{thm:A.9-in-paper} we obtain that the limit $\lim_{t\rightarrow +\infty} \widetilde{\Phi}(t,X(t),Y(t))$ exists and is finite a.s. This leads to the desired convergence rates.
\end{proof}
\begin{rmk}\label{rmk:different-theorems-for-different-b}
{\rm
Depending on $b$ and under stronger integrability conditions on the diffusion term, the convergence rate results in Theorem~\ref{thm:convergence-rates-shbf} may offer an improvement over those established in Theorem~\ref{thm:convergence-shbf}.

(a) For $r >1$, let
    \begin{align*}
        b(t) := t^r \quad \forall t \geq t_0.
    \end{align*}
For every $s \geq t_0$ it holds
    \begin{align*}
        \frac{1}{b(s)} \int_{\frac{s+t_0}{2}}^{s} b(u)du \leq \frac{s}{r+1},
    \end{align*}
thus, assuming that
    \begin{align*}
        \int_{t_0}^{+\infty} s\|\sigma_Y(s)\|_{HS}^2 ds < +\infty,
    \end{align*}
both integrability conditions in Theorem~\ref{thm:convergence-rates-shbf} are satisfied. Under this assumption, the convergence rates improve to
    \begin{align*}
        f(Y(t)) - \inf\nolimits_\mathcal{H} f = \mathcal{O}\left(\tfrac{1}{t^{r+1}}\right)  \ \mbox{and} \ \|X(t)\| = \mathcal{O}\left(\tfrac{1}{\sqrt{t}}\right) \ \mbox{as} \ t \to +\infty \ \mbox{a.s.}
    \end{align*}
    
(b) As a second example, we consider for $r >1$ and $t_0 >0$
    \begin{align*}
        b(t) := t^r \log(t) \quad \forall t \geq t_0.
    \end{align*}
For every $s \geq t_0$, it holds
    \begin{align*}
        \frac{1}{b(s)}\int_{\frac{s+t_0}{2}}^{s} b(u) du = & \ \frac{1}{s^r\log(s)} \int_{\frac{s+t_0}{2}}^{s} u^{r} \log(u)du = \frac{1}{s^r\log(s)} \left(\frac{u^{r+1}}{r+1}\log(u)\bigg\vert_{\frac{s+t_0}{2}}^{s} - \int_{\frac{s+t_0}{2}}^{s} \frac{u^{r}}{r+1}du\right)\\
        = & \ \frac{1}{s^r\log(s)} \left(\frac{s^{r+1}}{r+1}\log(s) - \frac{(s+t_0)^{r+1}}{2^{r+1}(r+1)}\log\left(\frac{s+t_0}{2}\right) - \frac{s^{r+1}}{(r+1)^2}+\frac{(s+t_0)^{r+1}}{2^{r+1}(r+1)^2}\right)\\
        \leq & \ s \ \mbox{for} \ s \ \mbox{large enough.}
\end{align*}
This shows that if
\begin{align*}
        \int_{t_0}^{+\infty} s\|\sigma_Y(s)\|_{HS}^2 ds < +\infty,
\end{align*}
then both integrability conditions in Theorem~\ref{thm:convergence-rates-shbf} are satisfied. Under this assumption, the convergence rates improve to
    \begin{align*}
        f(Y(t)) - \inf\nolimits_\mathcal{H} f = \mathcal{O}\left(\tfrac{1}{t^{r+1}\log(t)}\right)  \ \mbox{and} \ \|X(t)\| = \mathcal{O}\left(\tfrac{1}{\sqrt{t}}\right) \ \mbox{as} \ t \to +\infty \ \mbox{a.s.}
    \end{align*}

(c) In contrast, if $b(t) = \exp(t)$, no improvement in the convergence rates can be achieved via Theorem~\ref{thm:convergence-rates-shbf}. Namely,
\begin{align*}
        f(Y(t)) - \inf\nolimits_\mathcal{H} f = \mathcal{O}\left(\tfrac{1}{\exp(t)}\right)  \ \mbox{and} \ \|X(t)\| = \mathcal{O}\left(1\right) \ \mbox{as} \ t \to +\infty \ \mbox{a.s.}
    \end{align*}
This is not surprising, since for every $s \geq t_0$ we have
$$\frac{1}{b(s)}\int_{\frac{s+t_0}{2}}^{s} b(u) du = 1- \exp\left(\frac{t_0-s}{2}\right),$$
and therefore the second integrability assumption does not strengthen the required square integrability of the diffusion term.}\end{rmk}

\subsection{Application to the stochastic Su-Boyd-Cand\`{e}s system with vanishing damping}

In the following, we show -- via time-scaling arguments -- that a particular formulation of the stochastic Heavy Ball system \eqref{eq:SHBF} is closely related to the stochastic Su–Boyd–Cand\`{e}s system with vanishing damping \eqref{eq:SAVD-alpha}. This connection allows us to derive a long-time statements for the latter, not through a technically involved energy-dissipativity argument, but rather by transferring the results of Theorem~\ref{thm:convergence-shbf}.

\begin{thm}\label{thm:equivalence-shbf-savd}
Let $\alpha>1$. For given $t_0\geq 0$ and $s_0 > 0$, consider the dynamical systems
    \begin{align}\label{eq:SHBF-sp}\tag{SHBFEXP}
        \begin{cases}
        dY(t) = X(t) dt
        \\ dX(t) = \left(-X(t) - \left(\frac{s_0}{\alpha-1}\right)^2\exp\left(\frac{2(t-t_0)}{\alpha-1}\right) \nabla f(Y(t))\right) dt + \sigma_Y(t) dW(t)
        \\ Y(t_0) = Y_0,  \ X(t_0) = X_0
        \end{cases}
    \end{align}
on $[t_0, +\infty)$,    and
    \begin{align*}\tag{SAVD}
        \begin{cases}
            dZ(s) = Q(s)ds
            \\ dQ(s) = \left(-\frac{\alpha}{s}Q(s) - \nabla f(Z(s))\right) ds + \sigma_{Z}(s) dW(s)
            \\ Z(s_0) = Z_0, \ Q(s_0) = Q_0 \ 
        \end{cases}
    \end{align*}
on $[s_0, +\infty)$. Then, the following statements are true:
    \begin{enumerate} [(i)]
        \item  Let $(Y,X)$ be a strong solution of ~\eqref{eq:SHBF-sp}. Then, for
        \begin{align*}
            \tau(s) := t_0 + (\alpha-1)\log\left(\frac{s}{s_0}\right) \quad \forall s \geq s_0,
        \end{align*}
        it holds that $(Z,Q)$, defined by
        \begin{align*}
            Z(s) := Y(\tau(s)) \quad \mbox{and} \quad  Q(s) := \frac{\alpha-1}{s}X(\tau(s)),
        \end{align*}
        is a strong solution of~\eqref{eq:SAVD-alpha} with initial conditions $Z(s_0) = Z_0 =  Y_0$ and $Q(s_0) = Q_0 = \frac{\alpha-1}{s_0}X_0$, and diffusion term
        \begin{align*}
            \sigma_Z(s) = \left(\dot{\tau}(s)\right)^\frac{3}{2}\sigma_Y(\tau(s)) = \left(\frac{\alpha-1}{s}\right)^\frac{3}{2}\sigma_Y\left(\tau(s)\right).
        \end{align*}
        \item Let $(Z,Q)$ be a strong solution of~\eqref{eq:SAVD-alpha}. Then, for
        \begin{align*}
            \kappa(t) := s_0\exp\left(\frac{t-t_0}{\alpha-1}\right) \quad \forall t \geq t_0,
        \end{align*}
        it holds that $(Y,X)$, defined by
        \begin{align*}
            Y(t) = Z(\kappa(t)) \quad \mbox{and} \quad X(t) = \frac{s_0}{\alpha-1}\exp\left(\frac{t-t_0}{\alpha-1}\right)Q(\kappa(t))
        \end{align*}
        is a strong solution of~\eqref{eq:SHBF-sp} with initial conditions $Y(t_0) Y_0 = Z_0$ and $X(t_0) = X_0 = \frac{s_0}{\alpha-1} Q_0$, and diffusion term
        \begin{align*}
            \sigma_Y(t) = \left(\dot{\kappa}(t)\right)^\frac{3}{2}\sigma_Z(\kappa(t)) = \left(\frac{s_0}{\alpha-1}\exp\left(\frac{t-t_0}{\alpha-1}\right)\right)^\frac{3}{2}\sigma_Z(\kappa(t)).
        \end{align*}
    \end{enumerate}
\end{thm}
\begin{proof}
(i) Let $(Y,X)$ be a strong solution of ~\eqref{eq:SHBF-sp}. We denote
\begin{align*}
    b(t) := \left(\frac{s_0}{\alpha-1}\right)^2\exp\left(\frac{2(t-t_0)}{\alpha-1}\right).
\end{align*}
For $Z(s):=Y(\tau(s))$ and $ Q(s) := \frac{\alpha-1}{s}X(\tau(s))=\dot{\tau}(s) X(\tau(s))$, it yields, for every $s \geq s_0$, that
\begin{align*}
    dZ(s) = dY(\tau(s)) \dot{\tau}(s) ds = X(\tau(s))\dot{\tau}(s) ds = Q(s)ds,
\end{align*}
and
\begin{align*}
    dQ(s) &= \ddot{\tau}(s)X(\tau(s))ds + \dot{\tau}(s) d(X(\tau(s)))
    \\&= \ddot{\tau}(s) X(\tau(s)) ds + \dot{\tau}(s)^2\big(-X(\tau(s)) - b(\tau(s))\nabla f(Y(\tau(s)))\big)ds + \dot{\tau}(s)\sigma_Y(\tau(s)) dW(\tau(s))
    \\&= \ddot{\tau}(s) X(\tau(s)) + \dot{\tau}(s)^2\big(-X(\tau(s)) - b(\tau(s))\nabla f(Y(\tau(s)))\big)ds + \left(\dot{\tau}(s)\right)^\frac{3}{2}\sigma_Y(\tau(s)) dW(s)
    \\&= \left(\frac{\ddot{\tau}(s)}{\dot{\tau}(s)}Q(s) - \dot{\tau}(s)Q(s) - b(\tau(s)) \dot{\tau}(s)^2 \nabla f(Z(s))\right)ds + \left(\dot{\tau}(s)\right)^\frac{3}{2}\sigma_Y(\tau(s)) dW(s),
\end{align*}
where we used~\cite[Theorem 8.5.7]{oksendal} to justify the time transformation of the term integrated with respect to the Brownian motion. 

Since, for every $s \geq s_0$,
\begin{align*}
    \frac{\ddot{\tau}(s)}{\dot{\tau}(s)} - \dot{\tau}(s) =  -\frac{1}{s} - \frac{\alpha-1}{s} = -\frac{\alpha}{s},
\end{align*}
and
\begin{align*}
    b(\tau(s))(\dot{\tau}(s))^2 &= \left(\frac{\alpha-1}{ s}\right)^2\left(\frac{s_0}{\alpha-1}\right)^2\exp\left(\frac{-(\alpha-1) \log(s_0) + (\alpha-1)\log(s)}{\alpha-1}\right) = 1,
\end{align*}
we have
$$dQ(s) = \left(-\frac{\alpha}{s}Q(s) - \nabla f(Z(s)) \right) ds + \sigma_{Z}(s) dW(s).$$
Noting that conditions $(i)$ and $(ii)$ in the definition of a strong solution (see Definition \ref{def:strong-solution}) are trivially satisfied, we conclude that $(Z,Q)$ is a strong solution of \eqref{eq:SAVD-alpha}.

(ii) Conversely, let $(Z,Q)$ be a strong solution of~\eqref{eq:SAVD-alpha}. It is clear that $\tau:[s_0, +\infty) \to [t_0, +\infty)$ and $\kappa:[t_0, +\infty) \to [s_0, +\infty)$ are bijections and mutually inverse.

For $Y(t):=Z(\kappa(t))$ and $X(t) := \frac{s_0}{\alpha-1} \exp\left(\frac{t-t_0}{\alpha-1}\right)Q(\kappa(t)) =\dot{\kappa}(t)Q(\kappa(t))$, it yields, for every $t \geq t_0$, that
\begin{align*}
    dY(t) = dZ(\kappa(t)) \dot{\kappa}(t)dt= Q(\kappa(t)) \dot{\kappa}(t)dt = X(t) dt,
\end{align*}
and
\begin{align*}
 dX(t) &= \ddot{\kappa}(t)Q(\kappa(t)) + \dot{\kappa}(t) d(Q(\kappa(t)))
    \\
  &= \left(\ddot{\kappa}(t)Q(\kappa(t)) + \dot{\kappa}(t)^2\left(-\frac{\alpha}{\kappa(t)}Q(\kappa(t)) - \nabla f(Z(\kappa(t)))\right)\right)dt + \left(\dot{\kappa}(t)\right)^\frac{3}{2}\sigma_{Z}(\kappa(t)) dW(t)
    \\
    &= \left(\left(\ddot{\kappa}(t) - \frac{\alpha\dot{\kappa}(t)^2}{\kappa(t)}\right)\frac{1}{\dot{\kappa}(t)}X(t) - \dot{\kappa}(t)^2\nabla f(Y(t))\right)dt + \left(\dot{\kappa}(t)\right)^\frac{3}{2}\sigma_{Z}(\kappa(t)) dW(t)
    \\&=\left(\left(\frac{1}{\alpha-1} - \frac{\alpha}{\alpha-1}\right)X(t) - \dot{\kappa}(t)^2\nabla f(Y(t))\right) dt + \left(\dot{\kappa}(t)\right)^\frac{3}{2}\sigma_{Z}(\kappa(t)) dW(t)
    \\&= \left(- X(t) - \left(\frac{s_0}{\alpha-1}\right)^2\exp\left(\frac{2(t-t_0)}{\alpha-1}\right) \nabla f(Y(t))\right)dt + \sigma_Y(t) dW(t),
\end{align*}
since $\sigma_Y(t) = \left(\dot{\kappa}(t)\right)^\frac{3}{2}\sigma_Z(\kappa(t))$. The conditions $(i)$ and $(ii)$ in the definition of a strong solution (see Definition \ref{def:strong-solution}) are also in this case satisfied, therefore $(Y,X)$ is a strong solution of~\eqref{eq:SHBF-sp}.  
\end{proof}

We can now use the strong connection established in the above theorem to derive long-time statements for~\eqref{eq:SAVD-alpha}, by transferring those proved for~\eqref{eq:SHBF-sp} in Theorem \ref{thm:convergence-shbf}.

\begin{thm}\label{thm:convergence-savd}
    Let $\alpha>3$, let $s \mapsto s \sigma_Z(s)$ be square integrable, and let $(Z,Q) \in S^2_{\cal H \times \cal H}$ be a strong solution of~\eqref{eq:SAVD-alpha}. Then, the following statements are true:
 \begin{enumerate}[(i)]
        \item $f(Z(s)) - \inf_\mathcal{H} f = o\left(\tfrac{1}{s^2}\right)$ and $\|Q(s)\| = o\left(\tfrac{1}{s}\right)$ as $s\rightarrow +\infty$ \mbox{a.s.};
        \item $\mathbb{E} \left(f(Q(s)) - \inf_\mathcal{H} f\right)  = \mathcal{O}\left(\tfrac{1}{s^2}\right)$ and $\mathbb{E}\left(\|Q(s)\|^2\right) = \mathcal{O}\left(\tfrac{1}{s^2}\right)$ as $s\rightarrow +\infty$;
        \item $Z(s)$ converges weakly to a minimizer of $f$ as $s \rightarrow +\infty$ a.s.
    \end{enumerate}
\end{thm}
\begin{proof}
Let $(Z,Q) \in S^2_{\cal H \times \cal H}$ be a strong solution of~\eqref{eq:SAVD-alpha} and $t_0 \geq 0$.  Consistently with Theorem~\ref{thm:equivalence-shbf-savd}, we define, for every $t \geq t_0$,
\begin{align*}
        Y(t) := Z(\kappa(t)) \ \mbox{and} \ X(t) := \frac{\kappa(t)}{\alpha-1}Q(\kappa(t)),
\end{align*}
where $\kappa : [t_0, +\infty) \to [s_0, +\infty), \kappa(t) := s_0\exp\left(\frac{t-t_0}{\alpha-1}\right)$. It holds $\lim_{t \to +\infty} \kappa(t) = +\infty$. Then, $(Y,X) \in S^2_{\cal H \times \cal H}$ is a strong solution of~\eqref{eq:SHBF} for 
 \begin{align*}
       \lambda =1, \quad b(t) = \left(\frac{s_0}{\alpha-1}\right)^2\exp\left(\frac{2(t-t_0)}{\alpha-1}\right), \quad \mbox{and} \quad \sigma_Y(t) = \left(\dot{\kappa}(t)\right)^\frac{3}{2}\sigma_Z(\kappa(t)) \quad \forall t \geq t_0.
    \end{align*}
    
Since
\begin{align*}
        \dot{b}(t) = \frac{2 s_0^2}{(\alpha-1)^3}\exp\left(\frac{2(t-t_0)}{\alpha-1}\right) \quad \forall t \geq t_0.
    \end{align*}
assumption~\eqref{eq:assumption-b}, which reads
    \begin{align*}
        \sup_{t\geq t_0}\frac{\dot{b}(t)}{b(t)} < \lambda,
    \end{align*}
is equivalent to
\begin{align*}
        \frac{2}{\alpha-1} < 1 \Leftrightarrow \alpha >3.
\end{align*}
This explains the more restrictive assumption $\alpha>3$, which is in fact consistent with the one used in the deterministic setting for the Su–Boyd–Cand\`{e}s dynamical system.

Furthermore,  we observe that
    \begin{align*}
        \int_{t_0}^{+\infty} \|\sigma_Y(t)\|_{HS}^2 dt &= \int_{t_0}^{+\infty} \dot{\kappa}(t)^3\|\sigma_Z(\kappa(t))\|_{HS}^2 dt \stackrel{t=\tau(s)}{=} \int_{s_0}^{+\infty} \dot{\kappa}(\tau(s))^3\dot{\tau}(s)\|\sigma_Z(s)\|_{HS}^2 ds
        \\&= \int_{s_0}^{+\infty} \dot{\kappa}(\tau(s))^2\dot{(\kappa \circ \tau)}(s) \|\sigma_Z(s)\|_{HS}^2 ds = \int_{s_0}^{+\infty} \frac{s^2}{\alpha-1}\|\sigma_Z(s)\|_{HS}^2 ds <+\infty,
    \end{align*}
where $\tau = \kappa^{-1} : [s_0, +\infty) \to [t_0, +\infty), \tau(s):= t_0 + (\alpha-1)\log\left(\frac{s}{s_0}\right)$. It holds $\lim_{s \to +\infty} \tau(s) = +\infty$. In conclusion, all hypotheses of Theorem~\ref{thm:convergence-shbf} are satisfied.

We notice that, for every $s \geq s_0$,
\begin{align*}
        b(\tau(s)) = \left(\frac{s_0}{\alpha-1}\right)^2\exp\left(\frac{2(\tau(s)-t_0)}{\alpha-1}\right) = \left(\frac{s_0}{\alpha-1}\right)^2\exp\left(\frac{2(\alpha-1)\log\left(\frac{s}{s_0}\right)}{\alpha-1}\right) = \frac{1}{(\alpha-1)^2} s^2.
\end{align*}
and 
\begin{align*}
        \|Q(s)\| = \frac{\alpha-1}{\kappa(\tau(s))}\|X(\tau(s))\| = \frac{\alpha-1}{s}\|X(\tau(s))\|.
\end{align*}
    
From  Theorem~\ref{thm:convergence-shbf} (i), we have that $f(Y(\tau(s))) - \inf_\mathcal{H} f = o\left(\tfrac{1}{b(\tau(s))}\right)$ and $\|X(\tau(s))\| = o(1)$ as $s \rightarrow +\infty$ a.s. Therefore, $f(Z(s)) - \inf_\mathcal{H} f = o\left(\tfrac{1}{s^2}\right)$ and $\|Q(s)\| = o\left(\tfrac{1}{s}\right)$ as $s \rightarrow +\infty$ a.s. This proves (i). The two statements in (ii) follow analogously, by using Theorem~\ref{thm:convergence-shbf} (ii). Finally, by using Theorem~\ref{thm:convergence-shbf} (iii), we obtain that $Z(s) = Y(\tau(s))$ converges weakly to a minimizer of $f$ as $s \rightarrow +\infty$ a.s., therefore statement $(iii)$.
\end{proof}

\begin{rmk}{\rm
The technique used in the proof of the above theorem enabled us to recover the statements of \cite[Theorem 4.4]{mfao-stochasticintertialgradient} without resorting to a technically involved energy-dissipativity argument and without requiring the objective function $f$ to be twice continuously differentiable.
}
\end{rmk}

\section{The case of a monotone equation}\label{sec:operator}

In this section we will carry out a conceptually similar analysis for the stochastic Heavy Ball system \eqref{eq:SHBF-op} on $[t_0, +\infty)$
\begin{align*}
    \begin{cases}
        dY(t) = X(t)dt
        \\ dV(Y(t)) = H(t) dt
        \\ dX(t) = \big(-\lambda X(t) - \mu(t) H(t) - \gamma(t) V(Y(t))\big) dt + \sigma_Y(t) dW(t),
         \\ Y(t_0) = Y_0, \ \mbox{} \ X(t_0) = X_0,
    \end{cases}
\end{align*}
attached to the monotone equation \eqref{eq:mon}, where $X(\cdot)$, $Y(\cdot)$ and $V(Y(\cdot))$ are adapted stochastic processes with the property that there exists a continuous adapted process $H(\cdot)$ such that $dV(Y(t)) = H(t)dt$ for every $t \geq t_0$.

In the spirit of Definition \ref{def:strong-solution}, we consider as solutions of \eqref{eq:SHBF-op} pairs of trajectory processes $(Y,X) \in C([t_0, +\infty), {\cal H} \times {\cal H})$ such that, for every $t \geq t_0$,
\begin{align*}
    \begin{cases}
        Y(t) = Y_0 + \int_{t_0}^t X(s)ds\\ 
        V(Y(t)) = V(Y_0) + \int_{t_0}^t H(s)ds\\
        X(t) = X_0 + \int_{t_0}^t\big(-\lambda X(s) - \mu(s) H(s) - \gamma(s) V(Y(s))\big) ds + \int_{t_0}^t \sigma_Y(s) dW(s)
    \end{cases} \quad \mbox{almost surely.}
\end{align*}
The rather strong requirement that a continuous adapted process $H(\cdot)$ exists such that $dV(Y(t)) = H(t) dt$ for every $t \geq t_0$ can be met if $V$ is continuously differentiable, for example. Then, 
    \begin{align*}
      H(t) = \nabla V(Y(t)) X(t) \quad \forall t \geq t_0,
    \end{align*}
and its continuity is guaranteed in the case of a bounded diffusion term by Theorem \ref{thm:existence-uniqueness-solutions-shbf-op-alt}.

Since we do not assume that $V$ derives from a potential function, the convergence rate results will be formulated in terms of the residual $\|V(Y(t)\|$, the gap function $\langle Y(t)-y^*, V(Y(t))\rangle$, where $y^* \in \zer V$, and the velocity norm $\|X(t)\|$.

\subsection{Long-time analysis of \eqref{eq:SHBF-op}}

Firstly, we consider the existence and uniqueness of the trajectories defined by  \eqref{eq:SHBF-op}. For a trajectory solution $(Y,X)$ of \eqref{eq:SHBF-op}, defining, for every $t \geq t_0$,  
\begin{align*}
    Z(t) := X(t) + \mu(t)V(Y(t)).
\end{align*}
it yields
\begin{align*}
    dZ(t) = & \ \big (-\lambda X(t) - \mu(t)H(t) - \gamma(t)V(Y(t))\big)dt + \sigma_Y(t) dW(t) + \mu(t)H(t)dt + \dot{\mu}(t)V(Y(t)) dt\\
    = & \ \big(-\lambda Z(t) + (\lambda\mu(t) - \gamma(t) + \dot{\mu}(t))V(Y(t))\big)dt + \sigma_Y(t)dW(t).
\end{align*}
From the definition of $Z$ we have
\begin{align*}
    dY(t) = (Z(t) - \mu(t)V(Y(t))) dt \quad \forall t \geq t_0,
\end{align*}
which means that $(Y,Z)$ is a trajectory solution of the dynamical system on $[t_0, +\infty)$
\begin{align}\tag{SHBFOPALT}\label{eq:SHBF-op-alt}
    \begin{split}
        \begin{cases}
            dY(t) = (Z(t) - \mu(t) V(Y(t))) dt
            \\ dZ(t) = (-\lambda Z(t) + (\lambda\mu(t) - \gamma(t) + \dot{\mu}(t))V(Y(t)))dt + \sigma_Y(t) dW(t)
             \\ Y(t_0) = Y_0, \ \mbox{} \ Z(t_0) = X_0 + \mu(t_0)V(Y_0).
        \end{cases}
    \end{split}
\end{align}
Conversely, if $(Y,Z)$ is a trajectory solution of \eqref{eq:SHBF-op-alt} and there exists a continuous adaptive process $H(\cdot)$ such that
\begin{align*}
    dV(Y(t)) = H(t) dt \quad \forall t \geq t_0,
\end{align*}
then, by setting
\begin{align*}
    X(t) := Z(t) - \mu(t) V(Y(t)) \quad \forall t \geq t_0,
\end{align*}
we obtain that $(Y,X)$ is a trajectory solution of \eqref{eq:SHBF-op}. 

Using the Lipschitz continuity of $V$ and the continuity of $\mu$, we also have that, for every $\nu \geq 2$, $(Y,X)\in S_{\mathcal{H}\times\mathcal{H}}^\nu$ if and only if $(Y,Z)\in S_{\mathcal{H}\times\mathcal{H}}^\nu$.

In the following result we establish the existence and uniqueness of trajectory solutions of \eqref{eq:SHBF-op-alt}.

\begin{thm}\label{thm:existence-uniqueness-solutions-shbf-op-alt}
If $\|\sigma_Y(t)\|_{HS}\leq \sigma_*$ for every $t\geq t_0$, then ~\eqref{eq:SHBF-op-alt} has a unique strong solution $(Y,Z)$. In addition, $(Y,Z)\in S_{\mathcal{H}\times\mathcal{H}}^\nu$ for every $\nu \geq 2$.
\end{thm}
\begin{proof}
Let $t_0 \geq T$. We denote
    \begin{align*}
B:=L_V \sup_{t_0\leq t\leq T}\sqrt{\mu(t)^2 + (\lambda\mu(t) - \gamma(t) + \dot{\mu}(t))^2}.
\end{align*}

For $Y, Z \in C([t_0,T], {\cal H})$, $t_0 \leq t \leq T$, $\omega \in \Omega$ and $y^* \in \zer V$, it holds 
    \begin{align*}
        & \ \left\|\begin{pmatrix}
            Z(\omega,t) - \mu(t) V(Y(\omega,t))
            \\ -\lambda Z(\omega,t) + (\lambda\mu(t) - \gamma(t) + \dot{\mu}(t))V(Y(\omega,t))
        \end{pmatrix}\right\| + \|\sigma_Y(t)\|_{HS} \\
        = & \ \sqrt{\|Z(\omega,t) - \mu(t)V(Y(\omega,t))\|^2 + \|-\lambda Z(\omega,t) + (\lambda\mu(t) - \gamma(t) + \dot{\mu}(t))V(Y(\omega,t))\|^2} + \|\sigma_Y(t)\|_{HS}\\
        \leq & \ \sqrt{2(1+\lambda^2)\|Z(\omega,t)\|^2 + 2L_V^2\Big(\mu(t)^2 + (\lambda\mu(t) - \gamma(t) + \dot{\mu}(t))^2\Big)\|Y(\omega,t) - y^*\|^2} + \|\sigma_Y(t)\|_{HS}\\
        \leq & \ \left(\max\left(\sqrt{2(1+\lambda^2)}, 4B, 4B\|y^*\| \right) + \sigma_* \right) \left( 1+ \sup_{t_0 \leq s \leq T} \left\| \begin{pmatrix}
            Y(\omega,s)
            \\Z(\omega,s)
        \end{pmatrix}\right\|\right).
    \end{align*}
which shows that condition $(1)$ in Theorem~\ref{thm:existence-uniqueness-solutions} (i) is satisfied.
    
Now, let  $Y_1,Y_2,Z_1,Z_2\in C([t_0,T],\mathcal{H})$, $t_0\leq t\leq T$ and $\omega\in\Omega$ be fixed. It holds
    \begin{align*}
        &\left\| \begin{pmatrix}
            Z_1(\omega,t) - \mu(t) V(Y_1(\omega,t))
            \\ -\lambda Z_1(\omega,t) + \big(\lambda \mu(t) - \gamma(t) + \dot{\mu}(t)\big)V(Y_1(\omega,t))
        \end{pmatrix} - \begin{pmatrix}
            Z_2(\omega,t) - \mu(t)V(Y_2(\omega,t))
            \\ -\lambda Z_2(\omega,t) + \big(\lambda \mu(t) - \gamma(t) + \dot{\mu}(t)\big) V(Y_2(\omega,t))
        \end{pmatrix} \right\| \\
        & + \|\sigma_Y(t) - \sigma_Y(t)\|_{HS} \\
        \leq & \ \sqrt{2(1+\lambda^2)\|Z_1(\omega,t)-Z_2(\omega,t)\|^2 + 2L_V^2(\mu(t)^2 + (\lambda \mu(t) - \gamma(t) + \dot{\mu}(t))^2)\|Y_1(\omega,t)-Y_2(\omega,t)\|^2}\\
        \leq & \ \max \left(\sqrt{2(1+\lambda^2)}, \sqrt{2}B \right)
        \left\|\begin{pmatrix}
            Y_1(\omega,t)
            \\ Z_1(\omega,t)
        \end{pmatrix} - \begin{pmatrix}
            Y_2(\omega,t)
            \\ Z_2(\omega,t)
        \end{pmatrix}\right\|.
    \end{align*}
Hence, condition $(2)$ in Theorem~\ref{thm:existence-uniqueness-solutions} (i) is also satisfied, thus ~\eqref{eq:SHBF-op-alt} has a unique strong solution on $[t_0,T]$. Standard extension arguments guarantee the existence of a unique solution of ~\eqref{eq:SHBF-op-alt} on $[t_0,+\infty)$.
    
Arguing as above, according to Theorem~\ref{thm:existence-uniqueness-solutions} (ii), we obtain that for every $\nu \geq 2$ the solution $(Y,Z)$ of~\eqref{eq:SHBF-op-alt} lies in $S_{\mathcal{H}\times\mathcal{H}}^\nu[t_0,T]$ for every $T\geq t_0$, and thus in $S_{\mathcal{H}\times\mathcal{H}}^\nu$.
\end{proof}

For the long-time analysis of the trajectory solution of \eqref{eq:SHBF-op}, we shall assume that
\begin{align}\label{eq:assumptions-convergence-operator}
   \exists \lim_{t\rightarrow +\infty} \frac{\gamma(t)}{\mu(t)} =: \ell >0, \quad \sup_{t\geq t_0} \frac{\dot{\mu}(t)}{\gamma(t)} < 1 \quad \mbox{and} \quad 2\lambda - 3\ell + \inf_{t\geq t_0} \frac{\dot{\mu}(t)}{\mu(t)}>0.
\end{align}

\begin{rmk}
{\rm Assumption~\eqref{eq:assumptions-convergence-operator} implies in particular that $\lambda >\ell$. Indeed, let $\delta >0$ be such that  $2\lambda - \delta \geq 3\ell - \inf_{t\geq t_0}\frac{\dot{\mu}(t)}{\mu(t)}$. Since $\lim_{t\rightarrow +\infty} \frac{\gamma(t)}{\mu(t)} = \ell$, there exists $t_1 \geq t_0$ such that for every $t \geq t_1$ it holds $\ell-\frac{\delta}{2} \leq \frac{\gamma(t)}{\mu(t)} \leq \ell + \frac{\delta}{2}$, and in particular $\frac{\dot{\mu}(t)}{\mu(t)} \leq (\ell+\tfrac{\delta}{2})\frac{\dot{\mu}(t)}{\gamma(t)} \leq \ell+\tfrac{\delta}{2}$. Consequently, for every $t \geq t_1$ it holds
\begin{align*}
        2\lambda- \delta \geq 3\ell - \inf_{t\geq t_0}\frac{\dot{\mu}(t)}{\mu(t)} \geq 3\ell - \frac{\dot{\mu}(t)}{\mu(t)} \geq 3\ell - \ell - \tfrac{\delta}{2} = 2\ell - \tfrac{\delta}{2},
    \end{align*}
which implies $2\lambda > 2\ell + \tfrac{\delta}{2} > 2\ell$, and thus $\lambda > \ell$.}
\end{rmk}

\begin{thm}\label{thm:o-convergence-rates-operator}
Let $\sigma_Y$ be square integrable, let $(Y,X) \in S^2_{{\cal H} \times {\cal H}}$ be a trajectory solution of~\eqref{eq:SHBF-op}, and let $y^*$ be a zero of $V$. Then, the following statements are true:
    \begin{enumerate}[(i)]
    \item $\|V(Y(t))\| = o\left(\frac{1}{\mu(t)}\right)$, $\langle Y(t) - y^*, V(Y(t))\rangle = o\left(\frac{1}{\mu(t)}\right)$, and $\|X(t)\| = o(1)$ as $t \to +\infty$ a.s.;
    \item $\mathbb{E}(\|V(Y(t))\|^2) = \mathcal{O}\left(\frac{1}{\mu(t)^2}\right)$, $\mathbb{E}(\langle Y(t) - y^*, V(Y(t))\rangle) = \mathcal{O}\left(\frac{1}{\mu(t)}\right)$, and $\mathbb{E}(\|X(t)\|^2)=\mathcal{O}(1)$ as $t \to +\infty$ a.s.;
    \item $Y(t)$ converges weakly to a zero of $V$ as $t \to +\infty$ a.s.
    \end{enumerate}
\end{thm}
\begin{proof}
Choose $y^* \in \zer V$ and $0 < \eta < \lambda$. Consider the energy function
    \begin{align}
        \mathcal{E}_{\eta}(t,x,y,v) = & \ \frac{1}{2} \left\| 2\eta(y - y^*) + 2x + \mu(t) v\right\|^2 \label{eq:def-energy-shbf-op-line-1} \tag{$\mathcal{E}_\eta^1(t,x,y,v)$}
        \\& \ + 2\eta(\lambda-\eta) \|y - y^*\|^2 \label{eq:def-energy-sbhf-op-line-2}\tag{$\mathcal{E}_\eta^2(t,x,y,v)$}
        \\& \ + 2\eta\mu(t) \langle y - y^*, v\rangle \label{eq:def-energy-shbf-op-line-3}\tag{$\mathcal{E}_\eta^3(t,x,y,v)$}
        \\& \ + \frac{1}{2}\mu(t)^2\|v\|^2 \label{eq:def-energy-shbf-op-line-4} \tag{$\mathcal{E}_\eta^4(t,x,y,v)$}.
    \end{align}
We compute the derivatives
    \begin{align*}
        &\frac{d}{dt}\mathcal{E}_\eta^1(t,x,y,v) = \dot{\mu}(t)\langle v, 2\eta(y-y^*) + 2x + \mu(t)v\rangle,
        \\&\frac{d}{dx}\mathcal{E}_\eta^1(t,x,y,v) = 2(2\eta(y-y^*) + 2x + \mu(t)v),
        \\&\frac{d}{dy}\mathcal{E}_\eta^1(t,x,y,v) = 2\eta(2\eta(y-y^*) + 2x + \mu(t)v),
        \\&\frac{d}{dv}\mathcal{E}_\eta^1(t,x,y,v) = \mu(t)(2\eta(y-y^*) + 2x + \mu(t)v),
        \\&\frac{d^2}{dx^2}\mathcal{E}_\eta^1(t,x,y,v) = 4\Id,
        \\
        \\&\frac{d}{dy}\mathcal{E}_\eta^2(t,x,y,v) = 4\eta(\lambda-\eta)(y-y^*),
        \\
        \\&\frac{d}{dt}\mathcal{E}_\eta^3(t,x,y,v) = 2\eta\dot{\mu}(t)\langle y-y^*, v\rangle,
        \\&\frac{d}{dy}\mathcal{E}_\eta^3(t,x,y,v) = 2\eta\mu(t)v,
        \\&\frac{d}{dv}\mathcal{E}_\eta^3(t,x,y,v) = 2\eta\mu(t)(y-y^*)
        \\
        \\&\frac{d}{dt}\mathcal{E}_\eta^4(t,x,y,v) = \mu(t)\dot{\mu}(t)\|v\|^2,
        \\&\frac{d}{dv}\mathcal{E}_\eta^4(t,x,y,v) = \mu(t)^2 v.
    \end{align*}
The continuity of $(Y,X)$, the Lipschitz continuity of $V$ and the square integrability of the diffusion term allow us to apply It\^o's formula as stated in Proposition \ref{prop:ito-formula}. For every $t \geq t_0$, it holds
    \begin{align*}
        & \ d\mathcal{E}_\eta^1(t,X(t),Y(t),V(Y(t)))\\
        = & \ \langle 2\eta(Y(t)-y^*) + 2X(t) + \mu(t)V(Y(t)), \dot{\mu}(t)V(Y(t)) + 2\eta X(t) \rangle  dt\\
          & \ + \langle 2\eta(Y(t)-y^*) + 2X(t) + \mu(t)V(Y(t)), 2(-\lambda X(t) - \mu(t) H(t) - \gamma(t) V(Y(t))) + \mu(t) H(t)\rangle  dt\\
        & \ + \langle \sigma_Y(t)^*(2(2\eta(Y(t)-y^*) + 2X(t) + \mu(t)V(Y(t)))), dW(t)\rangle + 2 \trace(\sigma_Y(t)\sigma_Y(t)^*) dt\\
        = & \ \Big(4\eta(\eta-\lambda)\langle Y(t) - y^*, X(t)\rangle + 2\eta(\dot{\mu}(t) - 2\gamma(t))\langle Y(t) - y^*, V(Y(t))\rangle - 2\eta\mu(t) \langle Y(t) - y^*, H(t)\rangle\\
        & \quad + 4(\eta-\lambda)\|X(t)\|^2 + (2(\dot{\mu}(t)-2\gamma(t)) + 2(\eta-\lambda)\mu(t))\langle X(t), V(Y(t))\rangle - 2\mu(t)\langle X(t), H(t)\rangle\\
        &\quad + \mu(t)(\dot{\mu}(t) - 2\gamma(t))\|V(Y(t))\|^2 - \mu(t)^2\langle V(Y(t)), H(t)\rangle\Big)dt\\
        &\quad + \langle \sigma_Y(t)^*(2(2\eta(Y(t)-y^*) + 2X(t) + \mu(t)V(Y(t)))), dW(t)\rangle + 2 \|\sigma_Y(t)\|^2_{HS} dt,\\
        & \ d\mathcal{E}_\eta^2(t,X(t),Y(t),V(X(t))) = 4\eta(\lambda-\eta) \langle Y(t)-y^*, X(t) \rangle dt,\\
        & \ d\mathcal{E}_{\eta}^3(t,X(t),Y(t),V(Y(t))) \\
        = & \ (2\eta\dot{\mu}(t) \langle Y(t) - y^*, V(Y(t))\rangle + 2\eta \mu(t)\langle X(t), V(Y(t))\rangle + 2\eta\mu(t) \langle Y(t)-y^*, H(t)\rangle)dt,\\
        & \ d\mathcal{E}_\eta^4 (t,X(t),Y(t),V(Y(t))) = (\mu(t)\dot{\mu}(t)\|V(Y(t))\|^2 + \mu(t)^2 \langle V(Y(t)), H(t)\rangle) dt.
    \end{align*}
Summing these formulas gives, for every $t \geq t_0$,
    \begin{align*}
    \begin{split}
        & \ d\mathcal{E}_{\eta}(t,X(t),Y(t),V(X(t)))\\
        = & \ d\mathcal{E}_\eta^1(t,X(t),Y(t),V(Y(t))) + d\mathcal{E}_\eta^2(t,X(t),Y(t),V(Y(t)))\\
        & + d\mathcal{E}_\eta^3(t,X(t),Y(t),V(Y(t))) + d\mathcal{E}_\eta^4(t,X(t),Y(t),V(Y(t)))\\
        = & \ \Big(4\eta(\dot{\mu}(t) - \gamma(t))\langle Y(t)-y^*,V(Y(t))\rangle + 4(\eta-\lambda)\|X(t)\|^2  - 2\mu(t)\langle X(t), H(t)\rangle\\
        &\quad+ 2(2(\eta-\lambda)\mu(t) + \lambda \mu(t) - 2\gamma(t) + \dot{\mu}(t))\langle X(t),V(Y(t))\rangle + 2\mu(t)(\dot{\mu}(t) - \gamma(t))\|V(Y(t))\|^2\Big)dt
        \\&\quad + \langle \sigma_Y(t)^*(2(2\eta(Y(t)-y^*) + 2X(t) + \mu(t)V(Y(t)))), dW(t)\rangle + 2 \|\sigma_Y(t)\|^2_{HS} dt.
    \end{split}
    \end{align*}
We set
    \begin{align*}
        \varepsilon := \lambda-\eta >0.
    \end{align*}
    In the following analysis, it will be important to know the sign of
    \begin{align}\label{eq:part-of-energy-function-operator}
       t \mapsto -3\varepsilon\|X(t)\|^2 + 2(-2\varepsilon\mu(t) + \lambda\mu(t) - 2\gamma(t) + \dot{\mu}(t))\langle X(t), V(Y(t))\rangle + \frac{4}{3}\mu(t)(\dot{\mu}(t) - \gamma(t))\|V(Y(t))\|^2.
    \end{align}
 To accomplish this, we set
    \begin{align*}
        A(t):= -3\varepsilon, \quad B(t):= -2\varepsilon \mu(t) + \lambda\mu(t) - 2\gamma(t) + \dot{\mu}(t), \quad C(t):=\frac{4}{3}\mu(t)(\dot{\mu}(t) - \gamma(t))\|V(Y(t))\|^2.
    \end{align*}
It holds
    \begin{align*}
        B(t)^2-A(t)C(t) & = (-2\varepsilon\mu(t) + \lambda\mu(t) - 2\gamma(t) + \dot{\mu}(t))^2 + 4\varepsilon\mu(t)(\dot{\mu}(t) - \gamma(t))\\
        & = 4\varepsilon^2\mu(t)^2 - 4\varepsilon\mu(t)(\lambda\mu(t) - 2\gamma(t) + \dot{\mu}(t)) + (\lambda\mu(t) - 2\gamma(t) + \dot{\mu}(t))^2 + + 4\varepsilon\mu(t)(\dot{\mu}(t) - \gamma(t)) \\
        &= \mu(t)^2 \left(4\varepsilon^2 - 4 \varepsilon  \left(\lambda - \frac{\gamma(t)}{\mu(t)}\right) + \left(\lambda - \frac{2\gamma(t)}{\mu(t)} + \frac{\dot{\mu}(t)}{\mu(t)}\right)^2\right).
    \end{align*}
Our next aim is to identify an interval $I \subseteq (0, \lambda)$ such that 
the parabola $B^2-AC$ is negative for every $\varepsilon \in I$ and large enough $t$. To this end we use a technique introduced in \cite{attouch-bot-hulett-nguyen--heavy-ball}. 

From~\eqref{eq:assumptions-convergence-operator}, there exist $\delta_1,\delta_2$ such that
    \begin{align}\label{eq:introduce-delta1-and-delta2}
       \max \left (2-\frac{\lambda}{\ell}, \sup_{t \geq t_0} \frac{\dot{\mu}(t)}{\gamma(t)} \right )<\delta_1<\delta_2 < 1,
    \end{align}
which implies
    \begin{align*}
        &\lambda - \frac{2\gamma(t)}{\mu(t)} + \frac{\dot{\mu}(t)}{\mu(t)} < \lambda + (\delta_1 - 2)\frac{\gamma(t)}{\mu(t)} < \lambda + (\delta_2 - 2)\frac{\gamma(t)}{\mu(t)} < \lambda - \frac{\gamma(t)}{\mu(t)} \quad \forall t\geq t_0.
    \end{align*}
Using that $\lim_{t\rightarrow +\infty} \frac{\gamma(t)}{\mu(t)} =  \ell > 0$, there exist $\widetilde{\delta_1}, \widetilde{\delta_2}>0$ and $t_1\geq t_0$ and such that
    \begin{align}\label{eq:estimates-for-operator-rates-1}
 \lambda - \frac{2\gamma(t)}{\mu(t)} + \frac{\dot{\mu}(t)}{\mu(t)} < \lambda + (\delta_1 - 2)\ell < \widetilde{\delta}_1 < \widetilde{\delta}_2 < \lambda + (\delta_2 - 2)\ell < \lambda - \frac{\gamma(t)}{\mu(t)} \quad \forall t\geq t_1.
    \end{align}
On the other hand, we also know from~\eqref{eq:assumptions-convergence-operator} that
\begin{align*}
        \ell - \lambda < \lambda - 2\ell - \inf_{t \geq t_0} \frac{\dot{\mu}(t)}{\mu(t)}.
\end{align*}
We choose $\widetilde{\delta}_3, \widetilde{\delta}_4>0$ such that 
\begin{align*}
\ell - \lambda < -\widetilde{\delta}_4 < -\widetilde{\delta}_3 < \lambda - 2\ell + \inf_{t\geq t_0}\frac{\dot{\mu}(t)}{\mu(t)},
\end{align*}
and, using again that $\lim_{t\rightarrow +\infty} \frac{\gamma(t)}{\mu(t)} = \ell >0$, we obtain that there exists $t_2 > t_1$ such that
    \begin{align}\label{eq:estimates-for-operator-rates-2}
\frac{\gamma(t)}{\mu(t)} - \lambda < -\widetilde{\delta}_4 < -\widetilde{\delta}_3 < \lambda - \frac{2\gamma(t)}{\mu(t)} + \inf_{s\geq t_0}\frac{\dot{\mu}(s)}{\mu(s)} \leq \lambda - \frac{2\gamma(t)}{\mu(t)} + \frac{\dot{\mu}(t)}{\mu(t)} \quad \forall t\geq t_2.
    \end{align}
Combining~\eqref{eq:estimates-for-operator-rates-1} with~\eqref{eq:estimates-for-operator-rates-2}, we obtain
    \begin{align*}
        \frac{\gamma(t)}{\mu(t)} - \lambda < -\widetilde{\delta}_4 < -\widetilde{\delta}_3 < \lambda - \frac{2\gamma(t)}{\mu(t)} + \frac{\dot{\mu}(t)}{\mu(t)} < \widetilde{\delta}_1 < \widetilde{\delta}_2 < \lambda - \frac{\gamma(t)}{\mu(t)} \quad \forall t\geq t_2.
    \end{align*}
Without loss of generality, we assume
    \begin{align*}
        \max(\widetilde{\delta}_1,\widetilde{\delta}_3) = \widetilde{\delta}_1.
    \end{align*}
Then, the discriminant of the quadratic expression in $B(t)^2-A(t)C(t)$ is
    \begin{align}\label{eq:discriminant-parabola}
        \Delta_t := 16 \left(\left(\lambda-\frac{\gamma(t)}{\mu(t)}\right)^2 - \left(\lambda - \frac{2\gamma(t)}{\mu(t)} + \frac{\dot{\mu}(t)}{\mu(t)}\right)^2 \right) \geq 16(\widetilde{\delta}_2^2 - \widetilde{\delta}_1^2) > 0.
    \end{align}
Now, we choose $\delta_3>0$ such that
    \begin{align*}
        \delta_3 < \min\left( \ell,\lambda - \ell, \frac{\sqrt{\widetilde{\delta}_2^2 - \widetilde{\delta}_1^2}}{2}\right).
    \end{align*}
Using that $\lim_{t\rightarrow +\infty} \frac{\gamma(t)}{\mu(t)} = \ell$ once more yields the existence of $t_3\geq t_2$ such that
    \begin{align}\label{eq:help-operator-rates}
        \lambda - \ell -\delta_3 \leq \lambda - \frac{\gamma(t)}{\mu(t)} \leq \lambda - \ell + \delta_3 \quad \forall t\geq t_3.
    \end{align}
Going back to $B(t)^2-A(t)C(t)$, the roots of the quadratic form in this expression are given by
    \begin{align*}
        \underline{\varepsilon}_t := \frac{1}{2}\left(\lambda - \frac{\gamma(t)}{\mu(t)}\right) - \frac{\sqrt{\Delta_t}}{8} \ \mbox{and} \ \overline{\varepsilon}_t := \frac{1}{2}\left(\lambda - \frac{\gamma(t)}{\mu(t)}\right) +  \frac{\sqrt{\Delta_t}}{8}.
    \end{align*}
Using~\eqref{eq:discriminant-parabola} and~\eqref{eq:help-operator-rates}, we see that
    \begin{align*}
 \underline{\varepsilon}_t = \frac{1}{2}\left(\lambda - \frac{\gamma(t)}{\mu(t)}\right) - \frac{\sqrt{\Delta_t}}{8} < \frac{\lambda - \ell + \delta_3}{2} - \frac{\sqrt{\widetilde{\delta}_2^2 - \widetilde{\delta}_1^2}}{2} < \frac{\lambda - \ell -\delta_3}{2} \quad \forall t\geq t_3.
    \end{align*}
Similarly, we obtain
    \begin{align*}
        \frac{\lambda - \ell + \delta_3}{2} < \overline{\varepsilon}_t \quad \forall t\geq t_3.
    \end{align*}
In other words,
    \begin{align*}
        I:= \left[\frac{\lambda-\ell-\delta_3}{2}, \frac{\lambda-\ell+\delta_3}{2}\right] \subseteq (\underline{\varepsilon}_t,\overline{\varepsilon}_t)  \quad \forall t\geq t_3.
    \end{align*}
The way that $\delta_3$ was chosen also ensures that
    \begin{align*}
        I \subseteq (0,\lambda).
    \end{align*}
Therefore, for every $\varepsilon\in I$ and every $t \geq t_3$,
    \begin{align}\label{eq:negquad}
        B(t)^2 - A(t)C(t) < 0.
    \end{align}
    
In the following, we choose $\varepsilon \in I$ and set $\eta := \lambda - \varepsilon \in (0, \lambda)$. For every $t \geq t_3$, we have 
\begin{align*}
        d\mathcal{E}_\eta(t,Y(t),X(t),V(Y(t))) = & \ u(t) dt + \langle \sigma_Y(t)^*(2(2\eta(Y(t)-y^*) + 2X(t) + \mu(t)V(Y(t)))), dW(t)\rangle\\
        & \ + 2 \|\sigma_Y(t)\|^2_{HS} dt,
\end{align*}
where
\begin{align*}
u(t) := & \ 4\eta(\dot{\mu}(t) - \gamma(t))\langle Y(t)-y^*,V(Y(t))\rangle + 4(\eta-\lambda)\|X(t)\|^2  - 2\mu(t)\langle X(t), H(t)\rangle\\
        & \ + 2(2(\eta-\lambda)\mu(t) + \lambda \mu(t) - 2\gamma(t) + \dot{\mu}(t))\langle X(t),V(Y(t))\rangle + 2\mu(t)(\dot{\mu}(t) - \gamma(t))\|V(Y(t))\|^2 \\
        \leq & \ 4\eta \left(\dot{\mu}(t) - \gamma(t)\right)\langle Y(t) - y^*, V(Y(t))\rangle + (\eta-\lambda)\|X(t)\|^2 + \frac{2}{3}\mu(t)\left(\dot{\mu}(t)- \gamma(t)\right)\|V(Y(t))\|^2\\
        \leq & \ 0,
\end{align*}        
which follows from \eqref{eq:negquad}, together with the facts that $\langle X(t),H(t)\rangle \geq 0$ and $\langle Y(t) - y^*, V(Y(t))\rangle \geq 0$, both of which are consequences of the monotonicity of $V$.

Denoting
\begin{align*}
         U(t) := & \int_{t_3}^{t} (-u (s))ds, \quad A(t) := \int_{t_3}^{t} 2\|\sigma_Y(s)\|^2_{HS} ds,\\
        N(t) := & \int_{t_3}^{t} \langle \sigma_Y(s)^*(2(2\eta(Y(s)-y^*) + 2X(s) + \mu(s)V(Y(s)))), dW(s)\rangle,
    \end{align*}
we thus obtain, for every $t \geq t_3$,
\begin{align}\label{eq:expect}
        \mathcal{E}_\eta(t,Y(t),X(t),V(Y(t))) = \mathcal{E}_{\eta}(t_3,Y(t_3),X(t_3),V(Y(t_3))) - U(t) + N(t) + A(t).
\end{align}
The processes $A(\cdot)$ and $U(\cdot)$ are increasing and fulfill $A(t_3) = U(t_3) = 0$. According to Proposition~\ref{prop:ito-formula}, $N(\cdot)$ is a martingale, since $(Y,X) \in S^2_{{\cal H} \times {\cal H}}$ and for every $t \geq t_3$
\begin{align*}
        & \ \mathbb{E}\left(\int_{t_3}^{t}\left\|\sigma_Y(s)^*(2\eta(Y(s)-y^*) + 2X(s) + \mu(s)V(Y(s)))\right\|^2 ds\right) \\
        \leq & \ \mathbb{E}\left(\int_{t_3}^{t} 4 \|\sigma_Y(s)\|_{HS}^2\big((4\eta^2+2L_V^2\mu(s)^2) (\|Y(s)\|^2+\|y^*\|^2) + 4\|X(s)\|^2\big) ds\right)
        \\
        \leq & \ 4 \max\left(\sup_{t_3\leq s \leq t}(4\eta^2+2L_V^2\mu(s)^2), 4\right)\mathbb{E}\left(\sup_{t_3\leq s \leq t}\|X(s)\|^2 + \sup_{t_3\leq s \leq t}\|Y(s)\|^2 + \|y^*\|^2\right)\int_{t_3}^{t} \|\sigma_Y(s)\|_{HS}^2 ds < +\infty.
    \end{align*}
From the square-integrability of $\sigma_Y$, we have $\lim_{t\rightarrow+\infty} A(t)<+\infty$, therefore, by Theorem~\ref{thm:A.9-in-paper}, the limits $\lim_{t\rightarrow +\infty} \mathcal{E}_\eta(t,Y(t),X(t),V(Y(t))) <+\infty$ and  $\lim_{t\rightarrow +\infty} U(t)$ exist and are finite a.s.
    
Taking the expectation in~\eqref{eq:expect}, we obtain, for every $t\geq t_0$, that
    \begin{align*}
        \mathbb{E}\left(\mathcal{E}_\eta(t,Y(t),X(t), V(Y(t))) \right) &\leq \mathbb{E}(\mathcal{E}_\eta(t_0,Y(t_3),X(t_3), V(Y(t_3))) + 2 \mathbb{E}\left(\int_{t_3}^{t} \|\sigma_Y(s)\|^2_{HS} ds\right)
        \\&\leq \mathbb{E}(\mathcal{E}_\eta(t_0,Y(t_3),X(t_3), V(Y(t_3))) + 2 \mathbb{E}\left(\int_{t_0}^{+\infty} \|\sigma_Y(s)\|^2_{HS}ds\right) < +\infty.
\end{align*}
This means, in particular, that
    \begin{align*}
\sup_{t \geq t_0}\mathbb{E}\left(\|2 \eta(Y(t)-y^*) + 2X(t) + \mu(t) V(Y(t)) \|^2\right) <+\infty, \quad \sup_{t \geq t_0} \mathbb{E}\left(\|Y(t)-y^*\|^2\right) <+\infty,
    \end{align*}
    \begin{align*}
\sup_{t \geq t_0}\mathbb{E}\left(\mu(t) \langle Y(t)-y_*, V(Y(t)) \rangle \right) < +\infty \quad \mbox{and} \quad \sup_{t \geq t_0} \mathbb{E}\left(\mu(t)^2\|V(Y(t))\|^2\right) <+\infty.
    \end{align*}
Thus,
    \begin{align*}
        \sup_{t\geq t_0} \mathbb{E}(\|X(t)\|^2) <+\infty,
    \end{align*}
which proves (ii).
    
Recall that from \eqref{eq:introduce-delta1-and-delta2} and \eqref{eq:help-operator-rates} we have, for every $t\geq t_3$, that $\frac{\dot{\mu}(t)}{\gamma(t)}\leq \delta_1 < 1$ and $\ell - \delta_3 \leq \frac{\gamma(t)}{\mu(t)}\leq \ell + \delta_3$, thus
    \begin{align*}
(1-\delta_1)(\ell-\delta_3)\mu(t) \leq (1-\delta_1)\gamma(t) \leq \gamma(t) - \dot{\mu}(t).
    \end{align*}
Therefore, for every $t \geq t_3$,
    \begin{align*}
        & \ 4\eta(\ell-\delta_3)(1-\delta_1)\int_{t_3}^{t} \mu(s)\langle Y(s) - y^*, V(Y(s))\rangle ds + (\lambda - \eta) \int_{t_3}^{t} \|X(s)\|^2ds \\
        & + \frac{2}{3}(\ell-\delta_3)(1-\delta_1)\int_{t_3}^{t} \mu(s)^2\|V(Y(s))\|^2 ds + \frac{2}{3}\int_{t_3}^{t}\mu(s)\left(\gamma(s)-\dot{\mu}(s)\right)\|V(Y(s))\|^2ds \\
        \leq & \ 4\eta \int_{t_0}^{t} \left(\gamma(s) - \dot{\mu}(s)\right)\langle Y(s) - y^*, V(Y(s))\rangle ds + (\lambda - \eta)\int_{t_3}^{t} \|X(s)\|^2ds\\
        & \ + \frac{2}{3}\int_{t_3}^{t}\mu(s)\left(\gamma(s)-\dot{\mu}(s)\right)\|V(Y(s))\|^2ds\\
        \leq & \ \int_{t_3}^{t} (-u(s))ds  = U(t).
    \end{align*}
Given that $\lim_{t \to +\infty} U(t) < +\infty$ a.s., it follows that
    \begin{align*}
\int_{t_3}^{+\infty} \mu(t) \langle Y(t) - y^*, V(Y(t))\rangle dt < +\infty, \quad
         \int_{t_3}^{+\infty} \|X(t)\|^2dt < +\infty \quad \mbox{and} \quad
        \int_{t_3}^{+\infty} \mu(t)^2\|V(Y(t))\|^2 dt < +\infty \ \mbox{a.s.}
    \end{align*}
    
Now, we choose $\varepsilon_1,\varepsilon_2\in I$, $\varepsilon_1\neq \varepsilon_2$, and set $\eta_i := \lambda - \varepsilon_i, i=1,2$. For every $t \geq t_3$ it holds
    \begin{align*}
        & \ \mathcal{E}_{\eta_2}(t,Y(t),X(t),V(Y(t))) - \mathcal{E}_{\eta_1}(t,Y(t),X(t),V(Y(t))) \\
        = & \ \frac{1}{2}\left(\left\|2\eta_2(Y(t) - y^*) + 2X(t) + \mu(t)V(Y(t))\right\|^2 - \left\| 2\eta_2(Y(t) - y^*) + 2X(t) + \mu(t)V(Y(t))\right\|^2\right)\\
        & \ + 2(\eta_2(\lambda-\eta_2) - \eta_1(\lambda - \eta_1))\|Y(t) - y^*\|^2 + 2(\eta_2 - \eta_1) \mu(t)\langle Y(t) - y^*, V(Y(t))\rangle\\
        = & \ 2(\eta_2^2-\eta_1^2)\|Y(t)-y^*\|^2 + 2(\eta_2-\eta_1)\langle Y(t)-y^*, 2X(t) + \mu(t)V(Y(t))\rangle\\
        & \ + \left(2\lambda(\eta_2-\eta_1)-2(\eta_2^2-\eta_1^2)\right)\|Y(t)-y^*\|^2 + 2(\eta_2-\eta_1)\mu(t)\langle Y(t) - y^*, V(Y(t))\rangle\\
        = & \ 4(\eta_2-\eta_1) \left(\frac{\lambda}{2}\|Y(t)-y^*\|^2 + \left\langle Y(t)-y^*, X(t) + \mu(t)V(Y(t))\right\rangle\right).
    \end{align*}
We denote
    \begin{align*}
        p(t) := \frac{\lambda}{2}\|Y(t)-y^*\|^2 + \langle Y(t)-y^*, X(t) + \mu(t)V(Y(t))\rangle \quad \forall t \geq t_3.
    \end{align*}
Thus, since $\lim_{t\rightarrow +\infty}\mathcal{E}(t,Y(t),X(t),V(Y(t))) < +\infty$ exists a.s., it follows that $\lim_{t\rightarrow +\infty}p(t) <+\infty$ must exist a.s. With this in mind, we rewrite $\mathcal{E}_{\eta}(t,Y(t),X(t),V(Y(t)))$, for every $t \geq t_3$, as
    \begin{align*}
        \mathcal{E}_{\eta}(t,Y(t),X(t),V(Y(t))) = & \ \frac{1}{2}\left\| 2\eta (Y(t)-y^*) + 2X(t) + \mu(t)V(Y(t))\right\|^2 + 2\eta (\lambda-\eta)\|Y(t)-y^*\|^2\\
        & \ + 2\eta \mu(t)\langle Y(t)-y^*, V(Y(t))\rangle + \frac{1}{2}\mu(t)^2\|V(Y(t))\|^2\\
        = & \ 4\eta p(t) + \frac{1}{2}\|2X(t) + \mu(t) V(Y(t))\|^2 + \frac{1}{2}\mu(t)^2\|V(Y(t))\|^2\\
        = & \ 4\eta p(t) + \|X(t) + \mu(t)V(Y(t))\|^2 + \|X(t)\|^2.
    \end{align*}
    Define
    \begin{align*}
        h(t) := \|X(t) + \mu(t)V(Y(t))\|^2 + \|X(t)\|^2 \quad \forall t \geq t_3.
    \end{align*}
Since both $\lim_{t\rightarrow +\infty} \mathcal{E}_{\eta_1}(t,Y(t),X(t),V(Y(t))) < +\infty$ and $\lim_{t\rightarrow +\infty} p(t) < +\infty$ exist a.s., so does $\lim_{t\rightarrow +\infty} h(t) < +\infty$ exist. Since
    \begin{align*}
        \int_{t_3}^{+\infty} h(t)dt \leq 3\int_{t_3}^{+\infty} \|X(t)\|^2dt + 2\int_{t_3}^{+\infty} \mu(t)^2\|V(Y(t))\|^2dt < +\infty,
    \end{align*}
it yields that $\lim_{t\rightarrow+\infty} h(t) = 0$ a.s.
    Thus,
    \begin{align*}
        \lim_{t\rightarrow +\infty} \|X(t)\| =0 \ \mbox{a.s.},
    \end{align*}
    which can be combined with $\lim_{t\rightarrow +\infty} \|X(t) + \mu(t)V(Y(t))\|^2 = 0$ a.s. to deduce
    \begin{align*}
        \lim_{t\rightarrow +\infty} \mu(t)\|V(Y(t))\| = 0 \ \mbox{a.s.}
    \end{align*}
Using the boundedness of $t \mapsto \|Y(t)-y^*\|$, it yields
    \begin{align*}
        \lim_{t\rightarrow +\infty} \mu(t) \langle Y(t) - y^*, V(Y(t))\rangle = 0 \ \mbox{a.s.},
    \end{align*}
thus, (i) is proved.
    
It remains to show weak convergence of the trajectory $Y(t)$. Combining the fact that $\lim_{t \to +\infty} p(t) < +\infty$ and  $\lim_{t\rightarrow +\infty} \mu(t) \langle Y(t) - y^*, V(Y(t))\rangle = 0$, we obtain
$$\lim_{t \to +\infty} \left(\frac{\lambda}{2}\|Y(t)-y^*\|^2 + \langle Y(t)-y^*, X(t) \rangle \right) < +\infty \quad \mbox{a.s.}$$
Further, Lemma~\ref{lem:aux-convergence} guarantees that 
\begin{align*}
        \exists \lim_{t\rightarrow +\infty}\|Y(t)-y^*\| < +\infty \ \mbox{a.s.}
    \end{align*}
Hence, there exist subsets $\Omega_1,\Omega_{y*} \subseteq \Omega$ of full measure such that
    \begin{align*}
        \|V(Y(\omega,t))\| = o\left(\frac{1}{b(t)}\right) \ \mbox{} \forall \omega\in\Omega_1 \quad \mbox{and} \quad \lim_{t\rightarrow +\infty} \|Y(\omega,t) - y^*\| < +\infty \ \mbox{} \forall \omega\in\Omega_{y^*}.
    \end{align*}
Since $\zer V$ is a closed set, by using the separability of $\mathcal{H}$, we can show by repeating the arguments from the proof of Theorem \ref{thm:convergence-shbf} that there exists a countable set $S \subseteq \mathcal{H}$, which is dense in $\zer V$, such that $\mathbb{P}(\bigcap_{s\in S}\Omega_s \cap \Omega_1) = 1$ and, for every $\omega\in\bigcap_{s\in S} \Omega_s \cap \Omega_1$ and for every $y^*\in\zer V$, it holds
\begin{align*}
        \lim_{t\rightarrow +\infty} \|Y(\omega,t) - y^*\| < +\infty.
\end{align*}
Let $\Omega_\text{cont}\in\mathcal{F}$ be such that $\mathbb{P}(\Omega_\text{cont}) = 1$ and $Y(\omega, \cdot)$ is continuous for every $\omega\in\Omega_\text{cont}$, and  $\Omega_\text{converge} := \Omega_1 \cap \Omega_\text{cont} \cap \bigcap_{s\in S} \Omega_s$, which is also a subset of full measure of $\Omega$. 

We fix $\omega\in\Omega_\text{converge}$. Since $\lim_{t\rightarrow +\infty} \|Y(\omega,t) - y^*\| < +\infty$ exists, for every $y^*\in\zer V$, the first requirement for Opial's Lemma (\cite{opial}) is satisfied. As $Y(\omega, \cdot)$ is bounded, we can choose weak limit point $\bar y(\omega)$  of it. Using (i) and the fact that the graph of $V$ is sequentially closed in the weak$\times$strong topology of $\mathcal{H}\times\mathcal{H}$ (see~\cite{bauschke-combettes}), it yields  $V(\bar y(\omega)) = 0$. Therefore, $\bar y(\omega) \in \zer V$. This shows that the second condition of Opial's Lemma is also satisfied and therefore proves that $Y(\omega, t)$ converges weakly to a zero of $V$ as $t \to +\infty$.
\end{proof}

\subsection{Application to the stochastic Fast Optimistic Gradient Descent Ascent (OGDA) system}

Analogously to the minimization setting, we show -- using again time-scaling arguments -- that a particular formulation of \eqref{eq:SHBF-op} is closely connected to the stochastic variant of the Fast Optimistic Gradient Descent Ascent (OGDA) system recently introduced by Bo\c{t}, Csetnek, and Nguyen in \cite{bot-csetnek-nguyen}. This connection will also enable us to establish long-time results for the latter.

\begin{thm}\label{thm:transfer-results-to-SFOGDA-op}
Let $\alpha >1$ and $\beta >0$.  For given $t_0 \geq 0$ and $s_0 >0$, consider the dynamical systems
    \begin{align}\tag{SHBFOPEXP}\label{eq:SHBF-in-thm}
    \begin{cases}
        dY(t) = X(t)dt
        \\ dV(Y(t)) = H(t) dt
        \\ dX(t) = \left(-\frac{2(\alpha-1)}{\alpha} X(t) - \frac{2\beta s_0}{\alpha}\exp\left(\frac{2(t-t_0)}{\alpha}\right) H(t) - \frac{2\beta s_0}{\alpha}\exp\left(\frac{2(t-t_0)}{\alpha}\right) V(Y(t))\right) dt + \sigma_Y(t) dW(t)\\ 
         Y(t_0) = Y_0,  \ X(t_0) = X_0
    \end{cases}
\end{align}
on $[t_0, +\infty)$ and
\begin{align}\tag{SFOGDA}\label{eq:SFOGDA-in-thm}
    \begin{cases}
        dZ(s) = Q(s) ds\\
        dV(Z(s)) = K(s) ds\\
        dQ(s) = \left(-\frac{\alpha}{s}Q(s) - \beta K(s) - \frac{\alpha \beta}{2s} V(Z(s))\right) + \sigma_Z(s) dW(s)\\ 
         Z(s_0) = Z_0,  \ Q(s_0) = Q_0
    \end{cases}
\end{align}
on $[s_0, +\infty)$. Then, the following statements are true:
\begin{enumerate}[(i)]
    \item Let $(Y,X)$ be a trajectory solution of~\eqref{eq:SHBF-in-thm}. Then, for
    \begin{align*}
        \tau(s) := t_0 + \frac{\alpha}{2}\log\left(\frac{s}{s_0}\right)  \quad \forall s\geq s_0,
    \end{align*}
    it holds that $(Z,Q)$ defined by
    \begin{align*}
        Z(s):=Y(\tau(s)) \quad \mbox{and} \quad Q(s) := \frac{\alpha}{2s}X(\tau(s))
    \end{align*}
is a trajectory solution of~\eqref{eq:SFOGDA-in-thm} with initial conditions $Z(s_0) = Z_0 = Y_0$ and $Q(s_0)= Q_0 = \frac{\alpha}{2s_0}X_0$ and diffusion term
    \begin{align*}
        \sigma_Z(s) = \left(\dot{\tau}(s)\right)^\frac{3}{2}\sigma_Y(\tau(s)) = \left(\frac{\alpha}{2 s}\right)^\frac{3}{2}\sigma_Y(\tau(s)).
    \end{align*}
    \item Let $(Z,Q)$ be a trajectory solution of~\eqref{eq:SFOGDA-in-thm}. Then, for
    \begin{align*}
        \kappa(t) := s_0\exp\left(\frac{2(t-t_0)}{\alpha}\right) \quad \forall t\geq t_0,
    \end{align*}
    it holds that $(Y,X)$ defined by
    \begin{align*}
        Y(t):=Z(\kappa(t)) \quad \mbox{and} \quad X(t) := \frac{2s_0}{\alpha}\exp\left(\frac{2(t-t_0)}{\alpha}\right)Q(\kappa(t))
    \end{align*}
is a trajectory solution of~\eqref{eq:SHBF-in-thm} with initial conditions
$Y(t_0) = Y_0 =  Z_0$ and $X(t_0) = X_0 =  \frac{2s_0}{\alpha}Q_0$ and diffusion term
    \begin{align*}
        \sigma_Y(t) = \left(\dot{\kappa}(t)\right)^\frac{3}{2}\sigma_Z(\kappa(t)) = \left(\frac{2s_0}{\alpha}\exp\left(\frac{2(t-t_0)}{\alpha}\right)\right)^\frac{3}{2}\sigma_Z(\kappa(t)).
    \end{align*}
\end{enumerate}
\end{thm}
\begin{proof}
(i) Let $(Y,X)$ be a trajectory solution of~\eqref{eq:SHBF-in-thm}. We denote
    \begin{align*}
       \lambda := \frac{2(\alpha-1)}{\alpha} \ \mbox{and} \ \mu(t) := \frac{2\beta s_0}{\alpha}\exp\left(\frac{2(t-t_0)}{\alpha}\right).
    \end{align*}
For $Z(s):= Y(\tau(s))$, $Q(s) := \frac{\alpha}{2s} X(\tau(s)) = \dot{\tau}(s)X(\tau(s))$, and $K(s):=\dot{\tau}(s)H(\tau(s))$, it yields, for every $s \geq s_0$, that
    \begin{align*}
        dZ(s) = dY(\tau(s))\dot{\tau}(s)ds = X(\tau(s))\dot{\tau}(s)ds = Q(s)ds,
    \end{align*}
    and
    \begin{align*}
dQ(s) = & \ \ddot{\tau}(s)X(\tau(s))ds + \dot{\tau}(s)d(X(\tau(s)))\\
= & \ \ddot{\tau}(s) X(\tau(s))ds + (\dot{\tau}(s))^2\left(-\lambda X(\tau(s)) - \mu(\tau(s))H(\tau(s)) - \mu(\tau(s))V(Y(\tau(s)))\right)ds\\ 
& \ + \dot{\tau}(s)\sigma_Y(\tau(s))dW(\tau(s))\\
= & \ \ddot{\tau}(s)X(\tau(s))ds + (\dot{\tau}(s))^2(-\lambda X(\tau(s)) - \mu(\tau(s))H(\tau(s)) - \mu(\tau(s))V(Y(\tau(s))))ds\\
& \ + (\dot{\tau}(s))^\frac{3}{2}\sigma_Y(\tau(s)) dW(s)\\ 
= & \ \left(\frac{\ddot{\tau}(s)}{\dot{\tau}(s)}Q(s) - \lambda\dot{\tau}(s)Q(s)-\dot{\tau}(s)\mu(\tau(s))K(s) - (\dot{\tau}(s))^2\mu(\tau(s))V(Z(s))\right)ds\\
        & \ + (\dot{\tau}(s))^\frac{3}{2}\sigma_Y(\tau(s))dW(s),
    \end{align*}
where we used~\cite[Theorem 8.5.7]{oksendal} to justify the time transformation of the term integrated with respect to the Brownian motion. Since, for every $s\geq s_0$,
    \begin{align*}
        \frac{\ddot{\tau}(s)}{\dot{\tau}(s)} - \lambda\dot{\tau}(s) = \frac{1}{s} + \frac{\alpha-1}{s} = -\frac{\alpha}{s},
    \end{align*}
    and
\begin{align*}
        \dot{\tau}(s)\mu(\tau(s)) = \frac{\alpha}{2s}\frac{2\beta s_0}{\alpha}\exp\left(\frac{2\left(\frac{\alpha}{2}\log\left(\frac{s}{s_0}\right) + t_0 - t_0\right)}{\alpha}\right) =  \beta
    \end{align*}
as well as
    \begin{align*}
        (\dot{\tau}(s))^2\mu(\tau(s)) = \beta\dot{\tau}(s) = \frac{\beta\alpha}{2s},
    \end{align*}
we obtain that, for every $s \geq s_0$,
    \begin{align*}
        dQ(s) = \left(-\frac{\alpha}{s}Q(s) - \beta K(s) - \frac{\alpha}{2s}\beta V(Z(s))\right)ds + \sigma_Z(s) dW(s).
    \end{align*}
The continuity property is obviously inherited, therefore, $(Z,Q)$ is a trajectory solution of~\eqref{eq:SFOGDA-in-thm}.
    
(ii) Conversely, let $(Z,Q)$ be a trajectory solution of~\eqref{eq:SFOGDA-in-thm}. It is clear that $\tau:[s_0,+\infty)\rightarrow [t_0,+\infty)$ and $\kappa: [t_0,+\infty)\rightarrow [s_0,+\infty)$ are bijections and mutually inverse. 

For $Y(t):= Z(\kappa(t))$, $X(t):= \frac{2s_0}{\alpha}\exp\left(\frac{2(t-t_0)}{\alpha}\right)Q(\kappa(t)) = \dot{\kappa}(t)Q(\kappa(t))$ and $H(t):=\dot{\kappa}(t)K(\kappa(t))$, it yields, for every $t \geq t_0$, that
\begin{align*}
dY(t) = dZ(\kappa(t))\dot{\kappa}(t)dt = Q(\kappa(t))\dot{\kappa}(t)dt = X(t)dt,
\end{align*}
and
\begin{align*}
        dX(t) = & \ \ddot{\kappa}(t) Q(\kappa(t)) + \dot{\kappa}(t)d(Q(\kappa(t)))\\
        = & \ \ddot{\kappa}(t)Q(\kappa(t)) + (\dot{\kappa}(t))^2\left(-\frac{\alpha}{\kappa(t)}Q(\kappa(t)) - \beta K(\kappa(t)) - \frac{\alpha}{2\kappa(t)}\beta V(Z(\kappa(t)))\right)dt + (\dot{\kappa}(t))^\frac{3}{2}\sigma_Z(\kappa(t))dW(t)\\
        = & \ \left(\frac{\ddot{\kappa}(t)}{\dot{\kappa}(t)}X(t) - \frac{\alpha\dot{\kappa}(t)}{\kappa(t)}X(t) - \beta\dot{\kappa}(t)H(t) - \frac{\alpha\dot{\kappa}(t)^2}{2\kappa(t)}\beta V(Y(t))\right)dt +\sigma_Y(t)dW(t)\\
        = & \ \left(-\frac{2(\alpha-1)}{\alpha}X(t) - \frac{2\beta s_0}{\alpha}\exp\left(\frac{2(t-t_0)}{\alpha}\right)H(t) - \frac{2\beta s_0}{\alpha}\exp\left(\frac{2(t-t_0)}{\alpha}\right) V(Y(t))\right)dt + \sigma_Y(t)dW(t).
    \end{align*}
In this case as well, the continuity property is inherited, thus $(Y,X)$ is a trajectory solution of~\eqref{eq:SHBF-in-thm}.
\end{proof}

In the next theorem, we transfer the long-time statements for ~\eqref{eq:SHBF-in-thm}, as provided by Theorem \ref{thm:o-convergence-rates-operator}, to \eqref{eq:SFOGDA-in-thm}.

\begin{thm}\label{thm:rates-SFOGDA}
Let $\alpha > 2, \beta>0$, let $s \mapsto s \sigma_Z(s)$ be square integrable, and let $(Z,Q)\in S_{\mathcal{H}\times\mathcal{H}}^2$ be a trajectory solution of~\eqref{eq:SFOGDA-in-thm}, and let $z^*$ be a zero of $V$. Then, the following statements are true:
\begin{enumerate}[(i)]
    \item $\|V(Z(s))\| = o\left(\frac{1}{s}\right), \ \langle Z(s)-z^*, V(Z(s))\rangle = o\left(\frac{1}{s}\right)$, and $\|Q(s)\| = o\left(\frac{1}{s}\right)$ as $s\rightarrow +\infty$ a.s.;
    \item $\mathbb{E}(\|V(Z(s))\|^2) = \mathcal{O}\left(\frac{1}{s^2}\right), \ \mathbb{E}(\langle Z(s) - z^*, V(Z(s))\rangle) = \mathcal{O}\left(\frac{1}{s}\right)$, and $\mathbb{E}(\|Q(s)\|^2) = \mathcal{O}\left(\frac{1}{s^2}\right)$ as $s\rightarrow +\infty$;
    \item $Z(s)$ converges weakly to a zero of $V$ as $s\rightarrow +\infty$ a.s.
\end{enumerate}
\end{thm}

\begin{proof}
Let $(Z,Q)\in S_{\mathcal{H}\times\mathcal{H}}^2$ be a trajectory solution of~\eqref{eq:SFOGDA-in-thm} and $t_0\geq 0$. Consistently with Theorem~\ref{thm:transfer-results-to-SFOGDA-op}, we define
    \begin{align*}
        Y(t) := Z(\kappa(t)) \quad \mbox{and} \ X(t) := \frac{2}{\alpha}\kappa(t)Q(\kappa(t)),
    \end{align*}
    where $\kappa:[t_0,+\infty)\rightarrow[s_0,+\infty)$ is given by $\kappa(t) := s_0\exp\left(\frac{2(t-t_0)}{\alpha}\right)$. It holds $\lim_{t\rightarrow +\infty} \kappa(t) = +\infty$. Then, $(Y,X)\in S_{\mathcal{H}\times\mathcal{H}}^2$ is a trajectory solution of~\eqref{eq:SHBF-op} for
    \begin{align*}
        \lambda := \frac{2(\alpha-1)}{\alpha}, \ \mu(t)=\gamma(t):= \frac{2\beta s_0}{\alpha}\exp\left(\frac{2(t-t_0)}{\alpha}\right) \ \mbox{and} \ \sigma_Y(t) := \left(\frac{2s_0}{\alpha}\exp\left(\frac{2(t-t_0)}{\alpha}\right)\right)^\frac{3}{2}\sigma_Z(\kappa(t)) \ \forall t \geq t_0.
    \end{align*}
   
Notice that
    \begin{align*}
        \frac{\dot{\mu}(t)}{\gamma(t)} = \frac{\dot{\mu}(t)}{\mu(t)} = \frac{\left(\frac{\lambda s_0}{\alpha-1}\right)^2\frac{\lambda}{\alpha-1}\exp\left(\frac{\lambda(t-t_0)}{\alpha-1}\right)}{\left(\frac{\lambda s_0}{\alpha-1}\right)^2\exp\left(\frac{\lambda(t-t_0)}{\alpha-1}\right)} = \frac{\lambda}{\alpha-1} \quad \forall t \geq t_0.
    \end{align*}
Hence, assumption~\eqref{eq:assumptions-convergence-operator} is satisfied if and only if
    \begin{align*}
        \frac{\lambda}{\alpha-1}<1 \quad \mbox{and} \quad 2\lambda - 3 + \frac{\lambda}{\alpha-1} >0,
    \end{align*}
which is nothing else than
\begin{align*}
      \alpha >2.
    \end{align*}
This explains the more restrictive assumption $\alpha>2$, which is consistent with the one used in the deterministic setting for the Fast OGDA system.

It remains to show square integrability of $\sigma_Y$. We have
    \begin{align*}
        &\int_{t_0}^{+\infty} \|\sigma_Y(t)\|_{HS}^2 dt = \int_{s_0}^{+\infty} \dot{\kappa}(\tau(s))^3\dot{\tau}(s)\|\sigma_Z(s)\|_{HS}^2 ds
        \\= &  \int_{s_0}^{+\infty}\dot{\kappa}(\tau(s))^2\dot{(\kappa\circ\tau)}(s)\|\sigma_Z(s)\|^2ds = \int_{s_0}^{+\infty} \frac{4s_0^2}{\alpha^2}\frac{s^2}{s_0^2}\|\sigma_Z(s)\|^2ds = \frac{4}{\alpha^2}\int_{s_0}^{+\infty}s^2\|\sigma_Z(s)\|^2ds < +\infty,
    \end{align*}
    where $\tau = \kappa^{-1}:[s_0,+\infty)\rightarrow [t_0,+\infty)$, $\tau(s) = t_0 + \frac{\alpha}{2}\log\left(\frac{s}{s_0}\right)$. It holds $\lim_{s\rightarrow+\infty} \tau(s) = +\infty$. In conclusion, all hypotheses of Theorem~\ref{thm:o-convergence-rates-operator} are satisfied.
    
Recall that, for every $s \geq s_0$,
    \begin{align*}
        \mu(\tau(s)) = \mu\left(t_0 + \frac{\alpha}{2}\log\left(\frac{s}{s_0}\right) + t_0\right) = \frac{2\beta s}{\alpha},
    \end{align*}
so, according to Theorem~\ref{thm:o-convergence-rates-operator}(i), it holds
    \begin{align*}
        \|V(Z(s))\| = \|V(Y(\tau(s)))\| = o\left(\frac{1}{\mu(\tau(s))}\right) = o\left(\frac{1}{s}\right) \ \mbox{as} \ s\rightarrow +\infty \ \mbox{a.s.}
    \end{align*}
and
    \begin{align*}
\langle Y(\tau(s)) - z^*, V(Y(\tau(s)))\rangle = \langle Z(s) - z^*, V(Z(s))\rangle = o\left(\frac{1}{\mu(\tau(s))}\right) = o\left(\frac{1}{s}\right) \ \mbox{as} \ s\rightarrow +\infty \ \mbox{a.s.}
    \end{align*}
    Furthermore, it holds
    \begin{align*}
        \frac{2s}{\alpha}\|Q(s)\| = \|X(\tau(s))\| \rightarrow 0 \ \mbox{as} \ s\rightarrow +\infty \ \mbox{a.s.},
    \end{align*}
    which implies
    \begin{align*}
        \|Q(s)\| = o\left(\frac{1}{s}\right) \ \mbox{as} \ s\rightarrow +\infty \ \mbox{a.s.}.
    \end{align*}
The statements in (ii) follow analogously from Theorem~\ref{thm:o-convergence-rates-operator}(ii). Since $\tau(s)\rightarrow +\infty$ as $s\rightarrow +\infty$, 
we obtain that $Z(s) = Y(\tau(s))$ converges weakly to a zero of $V$ as $s \to +\infty$ a.s., which proves statement (iii).
\end{proof}

\begin{rmk}{\rm
The deterministic counterpart of \eqref{eq:SFOGDA} is the dynamical system
\begin{align*}
        \begin{cases}
            \ddot{Z}(s) + \frac{\alpha}{s}\dot{Z}(s) + \beta \frac{d}{ds}V(Z(s)) + \frac{\alpha \beta}{2s} V(Z(s)) = 0
            \\ Z(t_0) = Z_0 \ \mbox{and} \ \dot{Z}(t_0) = Q_0
        \end{cases}
    \end{align*}
which was introduced and studied in~\cite{bot-csetnek-nguyen} (in a slightly more general form) as the Fast Optimistic Gradient Descent Ascent (OGDA) dynamical system, in connection with solving monotone equations. The Fast OGDA dynamics achieve the best known convergence rates for monotone equations in the literature and reveal the surprising phenomenon that inertial dynamics with vanishing damping can have an accelerating effect beyond the setting of convex minimization. The same phenomenon has also been observed in~\cite{bot-csetnek-nguyen} for the explicit discretization of Fast OGDA, which yields an algorithm combining Nesterov momentum with a correction-term evaluation of the operator.
    }
\end{rmk}

\section*{Acknowledgements}
The authors would like to thank Julio Backhoff-Veraguas for fruitful discussions on various aspects of this work.

\renewcommand\thesection{\Alph{section}}
\renewcommand\thesubsection{\thesection.\Alph{subsection}}
\setcounter{section}{0}

\section{Appendix}

The following lemma plays a crucial in our long-time analysis.

\begin{lem}\label{lem:aux-convergence}
    Let $\lambda >0, t_0 \geq 0$ and $\phi:[t_0,+\infty) \rightarrow \R$ an It\^{o} process with $\phi(t) = \phi(t_0) + \int_{t_0}^{t} \psi(s) ds$ almost surely such that $\lim_{t\rightarrow +\infty} (\lambda\phi(t) + \psi(t)) = \ell \in \R$ almost surely. Then, $\lim_{t\rightarrow +\infty} \phi(t) = \frac{\ell}{\lambda}\in\R$ almost surely.
\end{lem}
\begin{proof}
    Let $\Omega_1$ be the subset of $\Omega$ of measure $1$ such that for all $\omega\in\Omega_1$ it holds $\lim_{t\rightarrow+\infty} \lambda\phi(\omega,t) + \psi(\omega,t) = \ell(\omega)$ as well as $\phi(\omega, t) = \phi(\omega, t_0) + \int_{t_0}^{t}\psi(\omega, s)ds$ for all $t\geq t_0$. We fix $\omega\in\Omega_1$ and set $\zeta(\omega,t) := \phi(\omega,t) - \frac{\ell(\omega)}{\lambda}$ for all $t \geq t_0$.  Hence, $\lim_{t\rightarrow +\infty} (\lambda \zeta(\omega,t) + \psi(\omega,t)) = 0$, and thus, for $\varepsilon >0$, there exists $T_\varepsilon \geq t_0$ such that for all $t\geq T_\varepsilon$ it holds
    \begin{align*}
        |\lambda\zeta(\omega,t) + \psi(\omega,t)| <\varepsilon,
    \end{align*}
therefore,
    \begin{align*}
        |\exp(\lambda t) (\lambda\zeta(\omega,t) + \psi(\omega,t))| < \varepsilon \exp(\lambda t).
    \end{align*}
Denoting
    \begin{align*}
        F(\ell,\phi,t) := \exp(\lambda t)\left(\phi - \frac{\ell}{\lambda}\right),
    \end{align*}
    we have
    \begin{align*}
    \frac{d}{dt} F(\ell,\phi,t) = \lambda \exp(\lambda t)\left(\phi - \frac{\ell}{\lambda}\right), \quad \frac{d}{d\phi} F(\ell,\phi,t) = \exp(\lambda t), \quad \frac{d}{d\ell} F(\ell,\phi,t) = -\frac{1}{\lambda}\exp(\lambda t).
    \end{align*}
By It\^{o}'s formula, we have for all $t \geq t_0$
    \begin{align*}
        dF(\ell,\phi(t),t) = \lambda\exp(\lambda t) \left(\phi(t) - \frac{\ell}{\lambda}\right) dt + \exp(\lambda t) d\phi(t) = (\lambda\exp(\lambda t)\zeta(t) + \exp(\lambda t) \psi(t)) dt.  
\end{align*}
From here it follows that for all $t \geq T_\varepsilon$
    \begin{align*}
        |\exp(\lambda t)\zeta(\omega,t) - \exp(\lambda T_\varepsilon)\zeta(\omega, T_\varepsilon)| 
        = & \ \left| \int_{T_\varepsilon}^t (\lambda \exp(\lambda s)\zeta(\omega, s) + \exp(\lambda s)\psi(\omega,s)) ds\right|
        \\
        \leq & \ \varepsilon \int_{T_\varepsilon}^t \exp(\lambda s) ds = \frac{\varepsilon}{\lambda}(\exp(\lambda t) - \exp(\lambda T_\varepsilon)),
    \end{align*}
    hence
    \begin{align*}
        \exp(\lambda t)|\zeta(\omega,t)| &\leq \left| \exp(\lambda t)\zeta(\omega,t) - \exp(\lambda T_\varepsilon)\zeta(T_\varepsilon)\right| + \exp(\lambda T_\varepsilon)|\zeta(\omega, T_\varepsilon)|
        \\& < \frac{\varepsilon}{\lambda} (\exp(\lambda t) - \exp(\lambda T_\varepsilon)) + \exp(\lambda T_\varepsilon)|\zeta(\omega, T_\varepsilon)|,
    \end{align*}
    and thus
    \begin{align*}
        |\zeta(\omega,t)| < \frac{\varepsilon}{\lambda} + \exp(\lambda(T_\varepsilon - t))\left(|\zeta(\omega,T_\varepsilon)| - \frac{\varepsilon}{\lambda}\right).
    \end{align*}
From here it yields
    \begin{align*}
        0\leq \limsup_{t\rightarrow +\infty} |\zeta(\omega,t)| \leq \frac{\varepsilon}{\lambda},
    \end{align*}
which completes the proof, since $\varepsilon>0$ and $\omega\in\Omega_1$ were chosen arbitrarily.
\end{proof}

Next, consider the general formulation of a stochastic differential equation on $[t_0, +\infty]$, where $t_0 \geq 0$,
    \begin{align*}
    \tag{SDEGEN}
    \label{SDE-gen}
        \begin{cases}
            dX(t)= F(t,X(t)) dt + G(t,X(t)) dW(t) 
            \\ X(t_0)=X_0 \in \mathcal{H},
        \end{cases}
    \end{align*}
defined on a filtered probability space $(\Omega,\mathcal{F},\{\mathcal{F}_t\}_{t\geq t_0},\mathbb{P})$, where $W$ denotes a $\mathcal{H}$-valued Brownian motion, $F: [t_0, +\infty) \times\mathcal{H} \to \mathcal{H}$ and $G: [t_0, +\infty) \times\mathcal{H} \to {\mathcal L}(\mathcal{H}, \mathcal{H})$. We equip ${\mathcal L}(\mathcal{H}, \mathcal{H})$ with the Hilbert-Schmidt inner product and corresponding norm $\|\cdot\|_{HS}$.

In the following definition, we introduce the notion of a solution that we consider for \eqref{SDE-gen}.

\begin{defi}\label{def:strong-solution}(\cite[Definition 3.1]{gawarecki-mandrekar})
    A stochastic process $X : \Omega \times [t_0, +\infty) \to \mathcal{H} $  is a strong solution of~\eqref{SDE-gen} if 
    \begin{enumerate}[(i)]
        \item $X\in C([t_0,+\infty),\mathcal{H})$ almost surely;
        \item for every $t \geq t_0$ it holds
        \begin{align*}
            \int_{t_0}^t \big(\|F(s,X(s))\| + \|G(s,X(s))\|_{HS}^2\big) ds < +\infty \quad \mbox{almost surely};
        \end{align*}
        \item for every $t \geq t_0$ it holds
        \begin{align*}
            X(t) = X(t_0) + \int_{t_0}^t F(s,X(s)) ds + \int_{t_0}^t G(s,X(s))dW(s) \quad \mbox{almost surely}.
        \end{align*}
    \end{enumerate}
\end{defi}

Furthermore, we introduce a class of stochastic processes characterized by a key property governing the expectation of their norm raised to a given power.

 \begin{defi}\label{def21}
(i) A stochastic process $X:\Omega \times [t_0, +\infty) \rightarrow \mathcal{H}$ is called progressively measurable, if for every $T \geq t_0$ the mapping
        \begin{align*}
            &\Omega\times [t_0,T] \rightarrow \mathcal{H}, \quad (\omega,s) \mapsto X(\omega,s),
        \end{align*}
        is $\mathcal{F}_T \otimes\mathcal{B}([t_0,T])$-measurable, where $\otimes$ denotes the product $\sigma$-algebra and $\mathcal{B}$ is the Borel $\sigma$-algebra. Further, $X$ is called adapted if $X(\cdot,T)$ is $\mathcal{F}_T$-measurable for every $T\geq t_0$.
        \item For $T \geq t_0$, we define the quotient space as
        \begin{align*}
            S_\mathcal{H}^0[t_0,T]:=\left\{X:\Omega\times[t_0,T]\rightarrow\mathcal{H}: \text{$X$ is a progressively measurable continuous stochastic process}\right\}\Big/\mathcal{R},
        \end{align*}
     and set $S_\mathcal{H}^0:=\bigcap_{T\geq t_0} S_\mathcal{H}^0[t_0,T]$. 
     
(ii) For $\nu>0$ and $T \geq t_0$, we define $S_\mathcal{H}^\nu [t_0,T]$ as the following subset of stochastic processes in $S_\mathcal{H}^0 [t_0,T]$
        \begin{align*}
            S_\mathcal{H}^\nu [t_0,T] := \left\{X\in S_\mathcal{H}^0 [t_0,T]: \quad \mathbb{E}\left(\sup_{t\in[t_0,T]} \|X(t)\|^\nu\right) < +\infty \right\}.
        \end{align*}
        Finally, we set $S_\mathcal{H}^\nu := \bigcap_{T\geq t_0} S_\mathcal{H}^\nu[t_0,T]$.
\end{defi}
The theorem below ensures the existence and uniqueness of solutions to \eqref{SDE-gen} within a fixed compact interval.

\begin{thm}\label{thm:existence-uniqueness-solutions}
Consider the general stochastic differential equation \eqref{SDE-gen} and let $T \geq t_0$. Assume that  $F$ and $G$ are jointly measurable, i.e. $F:[t_0,T]\times \mathcal{H}\rightarrow \mathcal{H}$ and $G:[t_0,T]\times \mathcal{H}\rightarrow \mathcal{L}(\mathcal{H},\mathcal{H})$ are  $\mathcal{B}([t_0,T])\otimes\mathcal{B}(\mathcal{H})$-measurable, and also measurable with respect to the product $\sigma$-field $\mathcal{B}([t_0,t])\otimes \mathcal{B}(\mathcal{H})$ on $[t_0,t]\times\mathcal{H}$ for every $t_0\leq t\leq T$.

\begin{enumerate}[(i)]
\item\label{thm:item:existence-uniqueness-strong-solution} (\cite[Theorem 3.3]{gawarecki-mandrekar}) Assume that the following conditions hold:
    \begin{enumerate}[(1)]
        \item There exists a constant $\ell$ such that for every $X \in C([t_0,T],\mathcal{H})$ almost surely it holds
        \begin{align*}
            \|F(t,X(\omega,t))\| + \|G(t,X(\omega,t))\|_{HS} \leq \ell \left(1+\sup_{t_0\leq s\leq T} \|X(\omega, s)\| \right) \ \mbox{for every} \ t_0 \leq t \leq T \ \mbox{and} \ \omega\in\Omega.
        \end{align*}
        \item For every $X,Y \in C([t_0,T],\mathcal{H}), t_0 \leq t \leq T$ and $\omega\in\Omega$, there exists $K > 0$ such that
        \begin{align*}
            \|F(t,X(\omega,t)) - F(t,Y(\omega,t))\| + \|G(t,X(\omega,t))-G(t,Y(\omega,t))\|_{HS} \leq K\sup_{t_0\leq s\leq T} \|X(\omega, s) - Y(\omega, s)\|.
        \end{align*}
    \end{enumerate}
 Then, ~\eqref{SDE-gen} has a unique strong solution on $[t_0,T]$.
    \item(\cite[Theorem A.7]{soto-fadili-attouch}, \cite[Theorem 5.2.1]{oksendal})\label{thm:item:help-existence-uniqueness} Additionally, assume that there exists $C>0$ such that
    \begin{align*}
        \|F(t,x)-F(t,y)\| + \|G(t,x)-G(t,y)\|_{HS} \leq C \|x-y\| \quad \mbox{for every} \ x,y\in\mathcal{H} \ \mbox{and} \ t_0 \leq t \leq T.
    \end{align*}
 Then, the unique solution of the stochastic differential equation~\eqref{SDE-gen} on $[t_0, T]$ lies in $S_{\mathcal{H}}^{\nu}[t_0,T]$ for every $\nu\geq 2$.
\end{enumerate}
\end{thm}

Using standard extension arguments, we then obtain a unique solution of the system \eqref{SDE-gen} on $[t_0, +\infty)$ which also belongs to $S_{\mathcal{H}}^{\nu}$ for every $\nu \geq 2$.

Next we recall the It\^o formula.

\begin{prop}\label{prop:ito-formula}\cite[Theorem 2.9]{gawarecki-mandrekar}
Let $T \geq t_0$ be fixed. Let $X$ be a stochastic process given by
\begin{align*}
    X(t) = X_0 + \int_{t_0}^{t} F(s) ds + \int_{t_0}^{t} G(s) dW(s) \quad \forall t\in[t_0,T],
\end{align*}
defined, as before, on a filtered probability space $(\Omega,\mathcal{F},\{\mathcal{F}_t\}_{t_0 \leq t \leq T},\mathbb{P})$, where $W$ denotes a $\mathcal{H}$-valued Brownian motion, $F:[t_0,T]\rightarrow \mathcal{H}$ and $G:[t_0,T]\rightarrow \mathcal{L}(\mathcal{H},\mathcal{H})$. Moreover, we require that $X_0$ is an $\mathcal{F}_{t_0}$-measurable $\mathcal{H}$-valued random variable, $F(\cdot)$ is adapted such that
\begin{align*}
    \int_{t_0}^{T} \|F(t)\| dt < +\infty \ \mbox{a.s.},
\end{align*}
and $G(\cdot)$ is adapted such that
\begin{align*}
    \int_{t_0}^{T} \|G(t)\|^2_{HS} dt < +\infty \ \mbox{a.s.}
\end{align*}
Let $\phi:[t_0, +\infty) \times\mathcal{H} \rightarrow \R$ be such that $\phi(\cdot,x)\in\text{C}^1([t_0, +\infty))$ for every $x\in\mathcal{H}$ and $\phi(t,\cdot)\in\text{C}^2(\mathcal{H})$ for every $t\geq t_0$. Then, 
    \begin{align*}
        \Tilde{X}(t):=\phi(t,X(t))
    \end{align*}
is an It\^{o} process such that
    \begin{align*}
        d\Tilde{X}(t) = \frac{d}{dt} \phi(t,X(t)) dt &+ \left\langle \nabla_x \phi(t,X(t)), F(t)\right\rangle dt + \langle \nabla_x\phi(t,X(t)),G(t) dW(t)\rangle
        \\&+ \frac{1}{2} \trace \left(\nabla_{xx}^2\phi(t,X(t)) G(t)G(t)^*\right) dt \quad \forall t \geq t_0.
    \end{align*}
    Here, $G(t)^*$ denotes the adjoint operator of $G(t)$ and $\trace(\cdot)$ the trace of an operator $A \in {\cal L}(\mathcal{H}, \mathcal{H})$, defined by
    \begin{align*}
        \trace(A) = \sum_{k=1}^\infty \langle Ae_k, e_k\rangle,
    \end{align*}
    where $\{e_k\}_{k\in\N}$ denotes an orthonormal basis of the Hilbert space $\mathcal{H}$.
    Further, if, for every $t_0 \leq t \leq T$,
    \begin{align*}
        \mathbb{E}\left(\int_{t_0}^t \|G(s)^*\nabla_x\phi(s,X(s))\|^2 ds\right)<+\infty,
    \end{align*}
    then $t \mapsto \int_{t_0}^t\left\langle G(s)^*\nabla_x\phi(s,X(s)), dW(s) \right\rangle$ is, for every $t_0 \leq t \leq T$, a square integrable continuous martingale with expected value $0$.
\end{prop}

The following theorem is central to the asymptotic analysis.

\begin{thm}(\cite[Theorem 3.9]{mao})\label{thm:A.9-in-paper}
    Let $A: [t_0, +\infty) \to \mathcal{H}$ and $U : [t_0, +\infty) \to \mathcal{H}$ be two continuous adapted increasing processes with $A(t_0)=U(t_0)=0$ almost surely. Let $N: [t_0, +\infty] \to \mathcal{H}$ be a real-valued continuous local martingale with $N(t_0)=0$ almost surely and $\xi$ be a nonnegative $\mathcal{F}_0$-measurable random variable. Define
    \begin{align*}
        X(t):=\xi + A(t) - U(t) + N(t) \quad \forall t\geq t_0.
    \end{align*}
    If $X(\cdot)$ is nonnegative and $\lim_{t\rightarrow +\infty}{A(t)} < +\infty$ almost surely, then the limits $\lim_{t\rightarrow +\infty}{X(t)} <+\infty$ and $\lim_{t\rightarrow +\infty}{U(t)} < +\infty$ exist almost surely.
\end{thm}

In the analysis of the dynamical system governed by an operator we shall make use of the following lemma.
\begin{lem}\label{lem:help-convergence-analysis-operator-case}
    Let $A,\ B,\ C \in \R$ be such that $A\neq 0$ and $B^2-AC \leq 0$. The following statements are true:
    \begin{enumerate}[(i)]
        \item If $A>0$, then it holds
        \begin{align*}
            A\|P\|^2 + 2B\langle P,Q\rangle + C\|Q\|^2 \geq 0 \quad \forall P,Q\in\mathcal{H};
        \end{align*}
        \item If $A<0$, then it holds
        \begin{align*}
            A\|P\|^2 + 2B\langle P,Q \rangle + C\|Q\|^2 \leq 0 \quad \forall P,Q\in\mathcal{H}.
        \end{align*}
    \end{enumerate}
\end{lem}

\printbibliography[]

\end{document}